\documentclass[10pt]{amsart}
\usepackage{amssymb}
\usepackage[all]{xy}
\usepackage{amsmath,amssymb, amsthm}

\def\id{{\rm id}}

\def\tm{\tilde{m}}

\newcommand{\norm}[1]{\| #1\|}

\def\b{          \beta}

\let\cal\mathcal

\def \R{{\mathbb R}}

\def \N{{\mathbb N}}

\newcommand{\prf}{{\begin{proof}}}
	\newcommand{\epf}{{\end{proof}}}

\newcommand{\CC}{{\mathbb C}}

\newtheorem{lemma}{\sc lemma}

\theoremstyle{definition}

\def\bee{\begin{equation}}
\def\eee{\end{equation}}

\theoremstyle{rema}

\newtheorem{rema}{\sc Remark}

\newcommand{\e}{\varepsilon}

\newcommand{\I}{\text{Im\,}}

\newcommand{\nb}{\overset{\circ}{\hat{B}}}
\newcommand{\A}{{\overset{\circ}{A}}}

\newcommand{\FF}{{\mathbb F}}
\newcommand{\RR}{{\mathbb R}}

\newcommand{\ZZ}{{\mathbb Z}}

\newcommand{\diag}{\mathop{\rm diag}}

\newtheorem{thm}{Theorem}

\newtheorem{corr}{Corollary}

\newtheorem{definition}{Definition}
\numberwithin{equation}{section}

\catcode`\@=12

\def\Empty{}
\newcommand\oplabel[1]{
  \def\OpArg{#1} \ifx \OpArg\Empty {} \else
  	\label{#1}
  \fi}
		
\newcommand{\comm}[1]{}
\newcommand{\comment}[1]{}

\begin{document}

\title[Density of positive Lyapunov exponents for symplectic cocycles]{Density of positive Lyapunov exponents for symplectic cocycles}

\author {Disheng Xu}
\address{
Univ Paris Diderot, Sorbonne Paris Cit\'e, Institut de Math\'ematiques de Jussieu-Paris Rive Gauche, UMR 7586, CNRS, Sorbonne Universit\'es, UPMC Univ Paris 06, F-75013, Paris, France
}
\email{disheng.xu@imj-prg.fr}

\date{\today}

\begin{abstract}
We prove that $Sp(2d,\RR)$, $HSp(2d)$ and pseudo unitary cocycles with at least one non-zero Lyapunov exponent are dense in all usual regularity classes for non periodic dynamical systems. For Schr\"odinger operators on the strip, we prove a similar result for density of positive Lyapunov exponents. It generalizes a result of A.Avila in \cite{avila2011density} to higher dimensions. 
\end{abstract}

\setcounter{tocdepth}{1}

\maketitle

\section{Introduction and main result}
Let $f:X\to X$ be a homeomorphism of a compact metric space, and $\mu$  a $f-$invariant probability measure on $X$, and $\FF$ be either $\RR$ or $\CC$. Suppose $A: X\to SL(n, \FF)$ or $GL(n,\FF)$ is a bounded measurable map, we can define the linear cocycle $(f,A)$ acting on $X \times \FF^n)$ as the following:
$$(x,y)\mapsto (f(x), A(x)\cdot y)$$

Some examples of linear cocycles are the derivative cocycle $(f, Df)$ of a $C^1-$ map of arbitrary dimensional torus, the random products of matrices, Schr\"odinger cocycles, etc. 

The main object of interest of linear cocycles is the asymptotic behavior of the products of $A$ along the orbits of $f$, especially the Lyapunov exponents. We consider the following definition.

The iterates of $(f,A)$ have the form $(f^n, A^n)$, where $$A^n(x):=\begin{cases}
A(f^{n-1}(x))\cdots A(x), n\geq 1\\
\id, n=0 \\
A(f^n(x))^{-1}\cdots A(f^{-1}(x))^{-1}, n\leq -1
\end{cases}$$
The top Lyapunov exponent for the cocycle $(f,A)$ is defined by 
\begin{equation}L_1(A)=L(A)=L(f,\mu, A)=\lim_{n\to \infty}\frac 1 n\int\ln\norm{A^n(x)}d\mu(x)
\end{equation}

The $k-$th Lyapunov exponent is defined as, \begin{equation}L_k(A):=\lim_{n\to \infty}\frac 1 n\int\ln\sigma_k(A^n(x))d\mu(x)
\end{equation}
where $\sigma_k(A)$ is the $k-$th singular value of $A$. We denote $L^k(A):=\sum_{j=1}^kL_j(A)$. The following remark gives the well-definedness of all the Lyapunov exponents:
\begin{rema}For $A\in GL(n,\FF)$, we can define its natural action, $\Lambda^k(A)$ on the space $\Lambda^k(\FF^n)$. 
	$$\Lambda^k(A)\cdot v_1\wedge \cdots \wedge v_k:=Av_1\wedge\cdots\wedge Av_k$$
	As a result, for a cocycle $(f,A)$ acting on $X\times  \FF^n$ we can define a cocycle $(f,\Lambda^k(A))$ on $(X,\Lambda^k(\FF^n))$. By Oseledec theorem the top Lyapunov exponent of cocycle $(f,\Lambda^k(A))$ is $L^k(A)$.
\end{rema}

We say the Lyapunov exponent of linear cocycle $A$ is positive if $L(A)>0$, the Lyapunov spectrum of $A$ is simple if 
\begin{equation*}
L_1(A)>\dots>L_n(A).
\end{equation*}
An important problem in the study of dynamical system is whether we can approximate a system by one with \textit{hyperbolic} behavior. In the setting of linear cocycles, we always assume the base dynamics $(f,\mu)$ is fixed and only the fiber dynamics $A$ should be allowed to vary. Then we ask whether the given linear cocycle can be approximated by one with positive Lyapunov exponents (or simple spectrum) in some regularity classes.

Basically, the study of Lyapunov exponents of linear cocycles depends on the regularity classes and base dynamics. When the base dynamics has hyperbolicity, and the cocycles are in general position of some higher regularity spaces, they can in most cases "borrow" some  hyperbolicity from the base dynamics. When $(f,\mu)$ is a random system, for example for products of random matrices, simplicity of Lyapunov spectrum was investigated in H.Furstenberg \cite{furstenberg1963noncommuting}, H.Furstenberg and H. Kesten \cite{furstenberg1960products}, Y.Guivarc'h and A.Raugi \cite{guivarc1985frontiere}, etc. In particular, when the support of the distribution of random matrices is Zariski dense, we have simple Lyapunov spectrum, see  I.Y.Gol'dsheid, G.A.Margulis's result in \cite{gol1989lyapunov} for example. And A.Avila and M.Viana \cite{avila2007simplicity} gave a criterion of simplicity of Lyapunov spectrum for linear cocycles over Markov map and proved the Zorich-Kontsevich conjecture.

When the ergodic system $(f,\mu)$ is hyperbolic, M.Viana \cite{viana2008almost} proved for any $s>0$, the set of $C^s-$cocycles with positive Lyapunov exponents are dense in $C^s-$topology. C.Bonatti and M.Viana \cite{bonatti2004lyapunov} proved the cocycles of simple Lyapunov spectrum are dense in the space of fiber bunched H\"older continuous cocycles. For weaker hyperbolicity assumption on base dynamics, i.e. partial hyperbolicity, using the techniques of partially hyperbolic systems, A.Avila, J.Santamaria, M.Viana \cite{avila2013holonomy} proved there is an open dense subset in the space of  fiber bunched H\"older continuous cocycles with positive Lyapunov exponents.

Of course there are many dynamical systems $f$ without any hyperbolicity. A typical one is  quasiperiodic systems. A dynamics system $(f,X,\mu)$ is called quasiperiodic if  $(f,X)$ is  an irrational rotation on torus preserving the Lebesgue measure $\mu$. The theory of Schr\"odinger and $SL(2,\FF)$-cocycles over quasiperiodic systems are extensively studied, see \cite{avila2009density}, \cite{avila2013monotonic}, \cite{avila2006reducibility}, \cite{avila2014complex}, \cite{duarte2014positive}, \cite{avila2015global} for example. In particular, A.Avila proved the \textit{stratified analyticity} of Lyapunov exponents and obtained a global theory for one frequency quasiperiodic analytic Schr\"odinger cocycles (see \cite{avila2015global}).

Notice that for the results above in general we need some regularity restriction for the cocycles. For example some bunching condition (maybe non-uniformly) and H\"older continuity are usually necessary for discussing cocycles over hyperbolic base system, and analyticity is a suitable assumption for cocycles over quasi-periodic systems. It is more difficult to study Lyapunov exponents for linear cocycles over general base systems and regularity class. Using a semi-continuity argument and Kotani theory in \cite{kotani1984ljaponov}, \cite{kotani1989jacobi}, A.Avila and D.Damanik proved that if the system $(f,\mu)$ is ergodic, then the set of cocycles with positive Lyapunov exponents is dense in $C(X,SL(2,\RR))$. A.Avila in \cite{avila2011density} extended this result to all usual regularity classes of $SL(2,\RR)-$cocycles using a local regularization formula proved by complexification method. 

\subsection{Main result for symplectic cocycles}
In this paper, we generalize the result in \cite{avila2011density} to symplectic, Hermitian-symplectic and pseudo unitary cocycles. 

\begin{definition}\label{sympl group def} The Symplectic group over $\FF$, denoted by $Sp(2d,\FF)$, is the group of all matrices $M\in GL(2d, \FF)$ satisfying $$M^TJM=J, \text{with } J=
\begin{pmatrix}0&I_d\\-I_d&0\end{pmatrix}.$$
The Hermitian-symplectic group $HSp(2d)$ is defined as:
	$$HSp(2d)=\{M\in GL(2d,\CC): M^\ast J M=J\}.$$
The pseudo unitary group $U(d,d)\subset GL(2d,\CC)$ is defined as: $$U(d,d):=\{A: A^\ast \begin{pmatrix}I_d&\\&-I_d\end{pmatrix}A=\begin{pmatrix}I_d&\\&-I_d\end{pmatrix}\}.$$
The special pseudo unitary group $SU(d,d)$ is defined as: $$SU(d,d):=\{A\in U(d,d), \det A=1\}.$$
The special Hermitian symplectic group $SHSp(2d)$ is defined as: $$SHSp(2d):=\{A\in HSp(2d), \det A=1 \}.$$
The Lie algebras of these groups are denoted respectively by $$\mathfrak{sp}(2d,\FF),\mathfrak{hsp}(2d), \mathfrak{u}(d,d), \mathfrak{su}(d,d), \mathfrak{shsp}(2d).$$
\end{definition}

As in \cite{avila2011density}, we have the following definition for \textit{ample} subspace of $C(X,G)$, where $G$ is a Lie group. Suppose $\mathfrak{g}$ is the Lie algebra of $G$.
\begin{definition}A topological space $\mathfrak B$ continuously included in $C(X, G)$ is ample if there exists a dense vector space $\mathfrak{b}\subset C(X, \mathfrak{g})$, endowed with a finer (than uniform) topological vector space structure, such that for every $A\in \mathfrak B, \exp(b)A\in \mathfrak B$ for all $b\in \mathfrak b$, and the map $b\mapsto \exp(b)A$ from $\mathfrak{b}$ to $\mathfrak B$ is continuous.
\end{definition}

\begin{rema}If $X$ is a compact smooth or analytic manifold, then the usual spaces of smooth or analytic maps $X\to G$ are ample in our sense.
\end{rema}

In this paper we prove the following theorem: 
\begin{thm}\label{mainresult}
	Suppose $f$ is not periodic on $supp(\mu)$, and let $\mathfrak{B}\subset C(X, Sp(2d,\RR))$ be ample. Then the set $\{A: L(A)>0\}$ is dense in $\mathfrak B$.
\end{thm}

\begin{corr}\label{main2}
	The same result in Theorem \ref{mainresult} holds if we replace $Sp(2d,\R)$ by $SHSp(2d)$, $SU(d,d)$, $HSp(2d)$ and $U(d,d)$.
\end{corr}

\subsection{Stochastic Schr\"odinger operators and Jacobi matrices on the strip} 
The most studied (Hermitian) symplectic cocycles are stochastic Schr\"odinger operator and Jacobi matrices on the strip, coming from the study of solid physics. For earlier studies of stochastic Jacobi matrices and Schr\"odinger operator on the strip and its relation to the Aubry dual of quasi-periodic Schr\"odinger operator, see \cite{damanik2004gordon}, \cite{haro2013thouless}, \cite{kotani1988stochastic}, \cite{geometry2010schulz} for example. 

Consider the following Jacobi matrices on the strip \footnote{for more general Jacobi matrices with matrices entries, see \cite{geometry2010schulz}, \cite{sadel2013herman} for example.}: 
\begin{eqnarray}
h_\omega: l^2(\ZZ, \CC^d)&\to& l^2(\ZZ, \CC^d) \nonumber\\
u&\mapsto& (h_\omega u)(n)=u(n+1)+u(n-1)+v_\omega(n)\cdot u(n)\label{jacobi strip}
\end{eqnarray}
where the potential $v_\omega(n)$ is a $d\times d$ Hermitian matrix, when $d=1$, it is $1-$dimensional Schr\"odinger operator (on the line). 

In this paper we always assume the potential is dynamically defined, i.e. $v_{(\cdot)}(n):=v(f^n(\cdot))$, $v:X\to Her(d) \text{ or } Sym_d\RR$ is a bounded measurable map, where $(f,X,\mu)$ is defined as in the beginning of the paper, and $Her(d)$, $Sym_d\FF$ are respectively the set of Hermitian matrices and symmetric $d\times d$ matrices over the field $\FF$. 

Then for energy $E$,  the corresponding eigenequation is the following: \begin{equation}\label{eigeneqn}
hu=Eu, \text{ with potential }v(f^n(x)).
\end{equation} Notice that for any $u: \ZZ\to \CC^d$ satisfies \eqref{eigeneqn}, we have 
\begin{equation}\begin{pmatrix}
u(n+1\\u(n)
\end{pmatrix}=\begin{pmatrix}
E-v(f^n(x))&-I_d\\
I_d
\end{pmatrix}\cdot \begin{pmatrix}
u(n)\\u(n-1)
\end{pmatrix}
\end{equation}
 Then the associated linear cocycle $(f,A^{(E-v)}): X\times \CC^n\to X\times\CC^n$ is defined by \begin{equation}A^{(E-v)}(x)=\begin{pmatrix}(E\cdot I_d-v(x))&-I_d\\ I_d&0\end{pmatrix}
\end{equation} Notice that $(f,A)$ is a (Hermitian) symplectic cocycle when $E\in \RR$. 

As in \cite{avila2011density}, we denote $L(A^{(E-v)})=L(E-v)$. Using a similar method to the proof of Theorem \ref{mainresult}, we prove the following result for (Hermitian) symplectic cocycles related to the stochastic Jacobi matrices on the strip with form in \eqref{jacobi strip} .

\begin{thm}\label{mainresultjacobi}Suppose $f$ is not periodic on $supp(\mu)$ and let $V\subset C(X,Her(d))$  or $C(X, Sym_d\RR)$ be a dense vector space endowed with a finer topological vector space structure. Then for any $E\in \RR$, the set of $v$ such that $L(E-v)>0$ is dense in $V$. 
\end{thm}

Now we study Schr\"odinger operator on the strip. Suppose $S\subset\ZZ^{\nu-1}$ is a finite connected set \footnote{As in \cite{kotani1988stochastic},  $S\subset \ZZ^{\nu-1}$ is a connected set means every two points of $S$ can be joint by a sequence of points in $S$, and any two consecutive points $p_n,p_{n+1}$ satisfy $|p_n-p_{n+1}|=1$ .}, we consider the following operator on $l^2(\ZZ\times S)$. 

\begin{equation}\label{sch on strip}
\tilde{u}\mapsto (h_\omega \tilde{u})(\alpha)=\sum_{|\beta-\alpha|=1,~\beta\in \ZZ\times S}\tilde{u}(\beta)+\tilde{v}_\omega(\alpha)\tilde{u}(\alpha)
\end{equation}
where for $p=(x_1,\dots, x_\nu), q=(y_1,\dots,y_\nu), |p-q|:=\sum_i|x_i-y_i|$. $\tilde{v}_\omega$ is a process ergodic under the one-dimensional group of translations.

As in \cite{kotani1988stochastic}, \eqref{sch on strip} can be viewed as an example of the stochastic Jacobi matrices on the strip with form in \eqref{jacobi strip}. For example if $S=\{1,..,d\}\subset \ZZ$, then the associated Jacobi matrices on strip are as follows: 
for $\tilde{u}$ satisfies \eqref{sch on strip}, let $u:\ZZ\to \CC^d$ such that  $$u(n)=(u_1(n),\dots, u_d(n)), u_i(n)=\tilde{u}(i,n)$$
then $u$ satisfies \eqref{jacobi strip} with potential $$v_\omega(n)=\begin{pmatrix}
\tilde{v}_\omega(n,1)&1&\\
1&\tilde{v}_\omega(n,2)& 1\\
&1&\tilde{v}_\omega(n,3)& 1\\
&&\ddots&\ddots&\ddots&\\
&&&1&\tilde{v}_\omega(n,k-1)&1\\
&&&&1&\tilde{v}_\omega(n,k)
\end{pmatrix}$$ 
For general $S$, $v_\omega$ always has entries of $\tilde{v}_\omega$ as diagonal elements and with non-random off-diagonal elements (be 0 or 1).

For Schr\"odinger operator on the strip defined in \eqref{sch on strip}, if $\#(S)=d$, and we consider the embedding  $\RR^d\hookrightarrow Sym_d\RR$ by identifying a vector with the diagonal elements of a symmetric matrix, then each measurable bounded map $v\in (X, \RR^d)\hookrightarrow(X, Sym_d\RR)$ induces a family of stochastic Schr\"odinger operator on strip, since the off-diagonal elements of the potential matrix are non-random. 

For potential $v: X\to \RR^d\hookrightarrow Sym_d\RR$ and energy $E$, We denote by $A^{(E-v)}$ the cocycle associated to the eigenequation $hu=Eu$ (with potential $v(f^n(x))$) of Schr\"odinger operator on the strip. Then we have the following similar result to Theorem \ref{mainresultjacobi}. 

\begin{corr}\label{mainresultsch}Suppose $f$ is not periodic on $supp(\mu)$ and let $V\subset C(X, \RR^d)$ be a dense vector space with a finer topological vector space structure. Then for any $E\in \RR$, the set of $v$ such that $L(A^{(E-v)})=L(E-v)>0$ is dense in $V$.
\end{corr}

\subsection{Main difficulty of the proof, novelty of the paper and some remarks}
The key step in the proof of our result is to prove Theorem \ref{m+=m-} in Chapter \ref{key chapter }. Basically speaking, it proves that, if the one parameter family of cocycles $R_\theta A$ satisfies $L(R_\theta A)=0$ for positive Lebesgue measure of $\theta$, then the function $m^+$ can determine $m^-$, where $m^+, m^-$ are defined as in Kotani theory and \cite{avila2013monotonic}, $R_\theta$ is the higher dimensional rotation.

Before that, we have to prove the monotonicity of \textit{fibered rotation function} of family $R_\theta A$, see section \ref{invariant cone field} and Lemma \ref{l/t=l'}. To prove it, we consider a cone field on the Lagrangian Grassmannian induced by the invariant cone on Lie algebra of Hermitian type. This idea is inspired by the study of Maslov index, see \cite{arnold2000complex} and \cite{freitasaction}. \footnote{
	After this paper was completed, I found my proof of monotonicity of fibered rotation function in Chapter \ref{rot mon} partially coincided with the work of Hermann Schulz-Baldes in \cite{rotation2007schulz} and the work of Roberta Fabbri, Russell Johnson and Carmen N\'u\~nez in \cite{rotation2003fabbri} . Thanks Qi Zhou for the references.}

The main difficulty in the proof of Theorem \ref{m+=m-} is to generalize Kotani theory to (Hermitian) symplectic cocycles. Basically the generalization of Kotani theory to Schr\"odinger cocycle on the strip is done by S.Kotani and B. Simon in \cite{kotani1988stochastic}. Our approach is for general symplectic cocycles and is closer to the work of monotonic cocycles in \cite{avila2013monotonic}.\footnote{In fact, we can define monotonicity for symplectic cocycles similarly. Using techniques in \cite{avila2013monotonic} and this paper, we can get similar results to \cite{avila2013monotonic} for symplectic cocycles, see \cite{liu2015monotonic}.} Considering the geometry of Hermitian symmetric spaces and $\Lambda^d(\CC^{2d})$, Theorem 5 will be further proved by some tricky calculations. 

In addition, in Appendix, using techniques of monotonic symplectic cocycles (see \cite{liu2015monotonic}) we generalize some results of periodic Schr\"odinger operators and $SL(2,\RR)$ cocycles in \cite{avila2009spectrum},\cite{avila2011density} to higher dimensions.

These ideas are partially inspired by the studies of monotonic cocycles \cite{avila2013monotonic}, higher dimensional cocycles in \cite{sadel2013herman} and \cite{avila2014complex}, the geometry of Hermitian symmetric space in \cite{clerc1998compressions} and \cite{freitasaction}. 

It is natural to ask whether the simplicity of Lyapunov spectrum holds for (Hermitian) symplectic cocycles in a dense set of some regularity set over general base dynamics. The difficulty to prove it is that when only some of the Lyapunov exponents coincides or vanishes we can not get too much dynamical information (for example, \textit{deterministic}, see \cite{kotani1988stochastic} or section \ref{density cont}) on the cocycles . It is also difficult to get the density of positive Lyapunov exponents for cocycles taking values in $SL(n,\FF)$ (except $SL(2,\RR)$), or even to get the simplicity of Lyapunov spectrum. The basic difficulty lies in the following, Kotani theory,  \cite{avila2011density} and our paper heavily depend on the fact that the associated groups act biholomorphically on some Hermitian symmetric spaces (see details in section \ref{hsspace}), which is not true for $SL(n,\FF)$ unless $SL(2,\RR)$. 

\subsection{The structure of the paper}

The outline of this paper is as follows: Chapter 2-6 are dedicated to the proof for Theorem \ref{mainresult}, in fact the proof can be easily adapted to prove the rest of the results in this paper.

Chapter 2 is a short introduction of the geometry of symplectic action on Siegel upper half plane and its boundary. 

Chapter 3 is dedicated to complexification of Lyapunov exponent and monotonicity of fibered rotation function, which implies an important equation in Lemma \ref{kotanimono}.

Chapter 4 is the proof for Theorem \ref{m+=m-}, which is the most difficult part in this paper. 

Chapter 5,6 are the rest of the proof of Theorem \ref{mainresult}, based on the arguments (deterministic, semi-continuity, finitely-many-valued cocycles, local regularization formula, etc.) for Schr\"odinger  and $SL(2,\RR)$ cocycles in \cite{avila2005generic},\cite{avila2011density},\cite{kotani1989jacobi},\cite{simon1983kotani},\cite{kotani1988stochastic}.

Chapter 7 shows how to adapt the proof of Theorem \ref{mainresult} to Hermitian symplectic group case, pseudo unitary group case and Schr\"odinger operator on the strip.

The appendix is for some properties of generic  periodic (Hermitian) symplectic cocycles which are used in the proof of Theorem \ref{generic+}.

\subsection*{Acknowledgement}
I would like to express my thanks to my director of thesis, Professor Artur Avila, for his supervision and useful conversations. I would like to thank Xiaochuan Liu for reading the first version of this paper, Qi Zhou for some useful suggestions for several versions. This research was partially conducted during the period when the author visited IMPA, supported by r\'eseau franco-br\'esilien en math\'ematiques.

\section{Geometry of the symplectic group action}
\subsection{Hermitian symmetric space, Bergman Shilov boundary and some notations }\label{hsspace}
Basically speaking, Kotani theory, techniques of monotonic cocycles and Avila's density result of Sch\"odinger and $SL(2,\RR)$-cocycles heavily depend on the fact that the group $SL(2,\RR)$ (or $SU(1,1)$) acts biholomorphically on Poincar\'e upper half plane (or disc) in $\CC$. In particular, $SL(2,\RR)$ and $SU(1,1)$ preserve corresponding Poincar\'e metric.

The higher dimensional extension of Poincar\'e upper half plane (or disc) are Hermitian symmetric spaces of non compact type. A Hermitian symmetric space is a Hermitian manifold which at every point has an inversion symmetry preserving the Hermitian structure. Hermitian symmetric space appears in the theory of automorphic forms and group representations. An example is the Siegel upper half plane and its disc model, our proof heavily depends on the fact that symplectic group can act on it isometrically.

Bergman discovered that, different from the one variable case, for a large class of domain, a holomorphic function of several variables is completely determined by its values on a proper closed subset of the topological boundary of the domain. We call the minimal one the Bergman-Shilov boundary or Shilov boundary.

In this paper, Shilov boundaries of Hermitian symmetric spaces that we are interested in are the set of unitary symmetric matrix $U_{sym}(\CC^d)$ and unitary group $U(d)$ which can be identified with real or complex Lagrangian Grassmannian.

We consider the following notations which will be used later.
\begin{definition}For a pair of complex $d\times d$ matrices $M,N$, we denote $M>N$ if $(M-N)^\ast=M-N$ and $M-N$ is positive definite.
\end{definition}
\begin{definition}
	For a square matrix $M$, we denote $Im(M):=\frac{M-M^\ast}{2i}$. For a complex number $z$, 
	$\mathfrak{R}(z):=\frac{z+\overline{z}}{2}, \mathfrak{I}(z):=\frac{z-\overline{z}}{2}$.
\end{definition}
\begin{rema}
	Later in Chapter \ref{complexification arg} we will define $\mathfrak{R}, \mathfrak{I}$ for an element in a complex Lie algebra $\mathfrak{g}^\CC$ with its real form $\mathfrak{g}$, which coincides with the definition here when $\mathfrak{g}=\RR, \mathfrak{g}^\CC=\CC$. 
\end{rema}

\begin{definition}
	We denote by $\norm{\cdot}_{HS}$ the Hilbert-Schmidt norm of matrix.
\end{definition}

\subsection{The symplectic action on the models of Siegel upper half plane}
We consider Siegel upper half plane and its disc model, which are the generalization of Poincar\'e upper half plane and Poincar\'e disc.
\begin{definition}The Siegel upper half plane $SH_d$ is defined as the following:$$SH_d:=\{X+iY\in Sym_d\CC, X,Y\in Sym_d\RR, Y>0\}$$
\end{definition}

\begin{definition}We define $SD_d$ as the set $$\{Z\in Sym_d\CC,  I_d-Z\bar Z>0 \}$$ Notice that $SD_d$ is the set of complex $d\times d$ symmetric matrices with operator norm less than $1$.
\end{definition}

Now we consider the symplectic action on $SH_d$ and $SD_d$. 
\begin{lemma}The symplectic group acts on the Siegel upper half plane transitively by the generalized M\"obius transformations:
	$$M=\begin{pmatrix}A&B\\C&D\end{pmatrix}\in Sp(2d,\RR), Z\in SH_d, M\cdot Z:=(AZ+B)(CZ+D)^{-1}$$ The stabilizer of the point $i\cdot I_d\in SH_d$ is $SO(2d, \RR)\cap Sp(2d, \RR)$.
\end{lemma}
\begin{proof}See \cite{freitasaction}.
\end{proof}

Consider the Cayley element 
\begin{equation}\label{Cayley ele}
C:=\frac{1}{\sqrt{2}}\begin{pmatrix}I_d& -i\cdot I_d\\ I_d& i\cdot I_d\end{pmatrix}
\end{equation} then for all $2d\times 2d$ complex matrix $A$, we denote $\A:=CAC^{-1}$. We have:

\begin{lemma}\label{caylay}(1). The map  $A\mapsto \A$ is a Lie group isomorphism from $Sp(2d,\RR)$ to $U(d,d)\cap Sp(2d, \CC)$.
	
	(2). The group $U(d,d)\cap Sp(2d, \CC)$ acts on the set $SD_d$ transitively by the generalized M\"obius transformations: suppose $M=\begin{pmatrix}A&B\\C&D\end{pmatrix}\in U(d,d)\cap Sp(2d, \CC)$ then
	$$Z\in SD_d, M\cdot Z:=(AZ+B)(CZ+D)^{-1}$$
	(3).The Cayley element induces a fractional transformation identifying $SH_d$ with $SD_d$, i.e. for $Z\in SH_d, \Phi_C(Z):=(Z-i\cdot I_d)(Z+i\cdot I_d)^{-1}$, we have the following commutative diagram:
	$$\xymatrix{
		SH_d\ar[rr]^{A}\ar[d]_{\Phi_C} & & SH_d \ar[d]^{\Phi_C} \\
		SD_d\ar[rr]^{\A} & & SD_d
	}$$
\end{lemma}

\begin{proof}See \cite{freitasaction}.
\end{proof}

Now we define the projective model for $SH_d$ and $SD_d$. Consider the complex Grassmannian $G_{2d,d}\CC$, the set of all $d-$dimensional subspaces of $\CC^{2d}$, and let $M_{2d, d}(\CC)$ be the space of all full rank $2d\times d$ complex matrices and view the columns of these matrices as a basis of a subspace of $\CC^{2d}$.

If we consider the action of $GL(d,\CC)$ by right multiplication on $M_{2d, d}(\CC)$, then the Grassmannian is $$G_{2d,d}=M_{2d,d}(\CC)/GL(d,\CC)$$
For each $\begin{pmatrix}A\\ B\end{pmatrix}$, we use $\begin{bmatrix}A\\ B \end{bmatrix}$ to represent the class of $\begin{pmatrix}A\\ B\end{pmatrix}$.
The projective model $SPH_d$ of $SH_d$ will be the set of all classes that admit a representative of the type 
$$\begin{pmatrix} Z\\ I_d \end{pmatrix} \text{ with } Z\in Sym_d\CC, Im(Z)>0$$

The group action on $SPH_d$ is the left matrix multiplication  by a representative of the class:$$\begin{pmatrix}A&B\\C&D\end{pmatrix} \cdot\begin{bmatrix}Z\\ I
_d\end{bmatrix}= \begin{bmatrix}AZ+B\\CZ+D\end{bmatrix}=\begin{bmatrix}(AZ+B)(CZ+D)^{-1}\\I_d\end{bmatrix}$$

The map connecting $SH_d$ to $SPH_d$ is
\begin{eqnarray*}
	SH_d &\to& SPH_d\\
	Z &\mapsto& \begin{bmatrix}Z\\ I_d\end{bmatrix}
\end{eqnarray*}

Similarly we can define the projective model $SPD_d$ of the disc $SD_d$ as the set of classes in $M_{2d,d}(\CC)$ that admit a representative of the type:

$$\begin{pmatrix} Z\\ I_d \end{pmatrix} \text{ with } Z\in Sym_d\CC, \norm{Z}<1$$
The symplectic action on $SPD_d$ and the identification between $SPD_d$ and $SD_d$ can be defined similarly.

\subsection{The boundaries of different models}\label{sec bdry}
\subsubsection{Stratification of finite and infinite boundaries}
All the properties in this subsection can be found in section 3 of \cite{freitasaction}.

Consider the boundary of $SD_d$ in $Sym_d\CC$. $$\partial SD_d=\{Z^T=Z, \norm{Z}=1\}$$

The M\"obius transform is well-defined on $\partial SD_d$. Moreover, it has a stratification, the strata are, for $1\leq k\leq d,$ $$\partial_kSD_d=\{Z\in \partial SD_d:\text{rank}(I-Z\overline{Z})=d-k\}$$
In particular, $\partial_dSD_d=U_{sym}(\CC^d)=U_d\cap Sym_d\CC$, which is the \textit{Shilov boundary} of $SD_d$, and it is an orbit of $U(d,d)\cap Sp(2d,\CC)-$action.

We can also take the closure of the Siegel upper half plane in $Sym_d\CC$,
$$\overline{SH_d}=\{Z\in Sym_d\CC: \text{Im}(Z)\geq 0\}$$
and then map it to $\partial SD_d$ using the extensions of the map $\Phi_C, \Phi_C^{-1}$ defined in Lemma \ref{caylay}. Notice that $\Phi_C^{-1}$ is not defined on the set
$$\{Z\in \partial{SD_d}, 1\in \text{ the spectrum of } Z\}$$
We call this set the \textit{infinite boundary} and its complement in $\partial SD_d$ the \textit{finite boundary}. 

The finite boundary contains a part of every stratum. We have the following property:
the image of the finite part of the stratum $\partial SD_d$ under the extension of $\Phi_C^{-1}$ is $$\text{fin}(\partial_kSH_d)=\{Z\in Sym_d\CC: \text{Im}(Z)\geq 0, \text{rank}(\text{Im}(Z))=d-k\}$$

\subsubsection{An atlas of Shilov boundary}
Consider $\text{fin}(\partial_dSH_d)=Sym_d\RR$, then $\Phi_C$ restricted to $Sym_d\RR$ gives a chart of $$\{Z\in \partial_dSD_d, 1\notin \text{ the spectrum of } Z\}$$ 

Similarly, for an element $g\in SL(2,R), g=\begin{pmatrix}a&b\\c&d\end{pmatrix}$, composite with the Caylay element, we get a chart of 
a dense  subset of $\partial_dSD_d$:
\begin{eqnarray*}\Phi_{Cg}: Sym_d\RR&\to& \{Z\in \partial_dSD_d, \frac{a-ic}{a+ic}\notin \text{ the spectrum of } Z\}\\
	Z&\mapsto&((a-ic)Z+(b-id))((a+ic)Z+(b+id))^{-1}
\end{eqnarray*}
As a result, if we pick a sequence of $g_k$ such that $\frac{a_k-ic_k}{a_k+ic_k}$ take more than $d$ different values, then the family $\{\Phi_{Cg_k}: Sym_d\RR\to \partial_dSD_d\}$ give an atlas for $\partial_dSD_d=U_{sym}(\CC^d)$.

\subsection{Invariant cone field and partial order}\label{invariant cone field}
\subsubsection{Invariant cone field}
We construct an invariant cone field $\mathcal{C}$ on Shilov boundary. For $\text{fin}(\partial_dSH_d)=Sym_d\RR$, we consider a cone field $\{h: h\in TSym_d\RR, h>0\}$. Then using the tangent map of $\Phi_{Cg_k}$ defined in last subsection, we get a cone field $\mathcal{C}$ on $TU_{sym}(\CC^d)$.

It is easy to check that the cone field $\{h>0\}$ is invariant under symplectic action. Therefore the cone field $\mathcal{C}$ is well-defined and invariant under $U(d,d)\cap Sp(2d,\CC)-$action.

\subsubsection{Partial order defined on the universal covering space}\label{partial order}
For $U_{sym}(\CC^d)$, consider its universal covering space $\widehat{U_{sym}(\CC^d)}$, denote the covering map by $\Pi:\widehat{U_{sym}(\CC^d)}\to U_{sym}(\CC^d)$.  Then we can lift the cone field $\mathcal{C}$ to a cone field $\widehat{\mathcal{C}}$ on $\widehat{U_{sym}(\CC^d)}$, which is invariant under the lift of the $Sp(2d,\CC)\cap U(d,d)-$action.

Using $\hat{\mathcal{C}}$, we define a partial order $"<"$ on $\widehat{U_{sym}(\CC^d)}$; we say $\hat{Z_0}< \hat{Z_1}$, if there is an $C^1$ path $p:[0,1]\to \widehat{U_{sym}(\CC^d)}$ such that
\begin{equation}p(0)=\hat{Z_0}, p(1)=\hat{Z_1}, p'(t)\in \hat{\mathcal{C}}(p(t))\label{order}
\end{equation}

For the determinant function restricted on $U_{sym}(\CC^d)$, we pick a continuous lift $\hat{\det}: \widehat{U_{sym}(\CC^d)}\to \RR$, such that  

\begin{equation}\label{def hatdet}
\pi\circ\hat{\det}=\det\circ\Pi 
\end{equation}
where  \begin{equation}\label{rtos1}\pi:\RR\to\mathbb{S}^1, \pi(x)=e^{ix}
\end{equation}
To check $"<"$ is actually a partial order and for later use, we have
\begin{lemma}\label{order on covering}
	\begin{enumerate}
		\item Suppose $\hat{Z}\in \widehat{U_{sym}(\CC^d)}$,  for a path $p: [0,1]\to \widehat{U_{sym}(\CC^d)}$ such that $p(0)=\hat{Z}, p'(0)\in \widehat{\mathcal{C}}(Z)$, we have $$\frac{d}{dt}|_{t=0}\hat{\det}(p(t))>0$$
		
		\item The order $"<"$ defined in \eqref{order} is a strict partial order, i.e. there is no $\hat{Z},\hat{Z_1},\hat{Z_2}\in \widehat{U_{sym}(\CC^d)}$ such that $\hat{Z}<\hat{Z}$ and $\hat{Z_1}<\hat{Z_2}<\hat{Z_1}$.

		\item For any $Z\in U_{sym}(\CC^d)$, any continuous lift of the path $\theta\mapsto e^{2i\theta}Z$ is monotonic with respect to the order $"<"$.
		
		\item Any lift of the $Sp(2d,\CC)\cap U(d,d)-$action preserves the order $"<"$.
	\end{enumerate}
\end{lemma}
\begin{proof}
	Suppose $1\notin \text{ the spectrum of }Z=\Pi(\hat{Z})$, by computation for $D\Pi(p'(0)):=H\in \cal{C}(Z)$, there is an $h\in Sym_d\RR, h>0 $ such that $H=-i(Z-1)h(Z-1)$. Then 
	\begin{eqnarray*}\det(\Pi(p(t)))&=&\det(Z+tH+o(t))\\
		&=&\det(Z-it(Z-1)h(Z-1)+o(t))\\
		&=&\det(Z)\det(1-it(1-Z^\ast)h(Z-1)+o(t))\\
		&&(\text{ }Z\text{ is a unitary matrix.})\\
		&=&\det(Z)\det(1+it(1-Z^\ast)h(1-Z)+o(t))
	\end{eqnarray*}
	notice that $(1-Z^\ast)h(1-Z)$ is positive definite, lift to the covering space we have $\frac{d}{dt}|_{t=0}\hat{\det}(p(t))>0$. 
	
	In the case $1\in \text{ the spectrum of }Z$, we can get other expression of the tangent vectors in $\cal{C}(Z)$ by $\Phi_{Cg_k}$, and the proof is similar. In summary we get the proof of (1). 
	
	As a corollary, we have 
	\begin{equation}\text{if } \hat{Z_1}<\hat{Z_2} \text { then }\hat{\det}(\hat{Z_1})<\hat{\det}(\hat{Z_2})\label{det mono}
	\end{equation} 
	which implies (2).
	
	For (3). by taking the derivative, we need to prove $iZ\in \cal{C}(Z)$. We only prove it in the case $1\notin \text{ the spectrum of }Z$, for other cases the proof is similar. 
	
	By computation, $\cal{C}(Z)=\{-i(Z-1)h(Z-1), h>0, h\in Sym_d\RR\}$. Take $h=-Z(1-Z)^{-2}$, then it can be checked that $h\in Sym_d\RR, h>0$, and $-i(Z-1)h(Z-1)=iZ$, which implies $iZ\in \cal{C}(Z)$.
	
	(4). is the corollary of invariance of the cone field under the lift of the $Sp(2d,\CC)\cap U(d,d)-$action.
\end{proof}
\subsection{Bergman metric and the volume form on $SD_d$.}
In this section, we define the Bergman metric on $SD_d$ which is a generalization of Poincar\'e metric on the Poincar\'e disc. (see \cite{krantz1992functions} for example) In particular, the symplectic group action preserves the Bergman metric.

\begin{definition}Let $D$ be a bounded domain of $\CC^n$, $d\lambda$ be the Lebesgue measure on $\CC^n$, let $L^2(D)$ be the Hilbert space of square integrable functions on $D$, and let $L^{2,h}(D)$ denote the subspace consisting of holomorphic functions in $D$, the $L^{2,h}(D)$ is closed in $L^2(D)$. 
	
	For every $z\in D$, the evaluation $ev_z: f\mapsto f(z)$ is a continuous linear functional on $L^{2,h}(D)$. By the Riesz representation theorem, there is a function $\eta_z(\cdot)\in L^{2,h}(D)$ such that $$ev_z(f)=\int_Df(\zeta)\overline{\eta_z(\zeta)}d\lambda(\zeta)$$
	The Bergman kernel $K$ is defined by $K(z,\zeta)=\overline{\eta_z(\zeta)}$.
\end{definition}

\begin{definition}Let $D\subset \CC^n$ be a domain and let $K(z,w)$ be the Bergman kernel on $D$, consider a Hermitian metric on the tangent bundle of $T_z\CC^n$ by $$g_{ij}(z):=\frac{\partial^2}{\partial z_i\partial \bar z_j}\log K(z,z)$$for $z\in D$. Then the length of a tangent vector $\xi\in T_z\CC^n$ is given by $$\norm{\xi}_{B,z}:=\sqrt{\sum_{i,j=1}^n g_{ij}(z)\xi\bar\xi_j}$$ This metric is called Bergman metric on $D$.
\end{definition}

We denoted by $d$ the Bergman metric on $SD_d$ . We have the following lemma for $d$.
\begin{lemma}\label{schw}(1).For $A\in Sp(2d,\RR), Z_1,Z_2\in SD_d$, $$d(\A Z_1,\A Z_2)=d(Z_1,Z_2).$$
	(2).For $t\in (0,1)$, $t\cdot SD_d:=\{tZ, Z\in SD_d \}$ is a bounded precompact set under metric $d$. And we have 
	$$d(tZ_1, tZ_2)\leq td(Z_1,Z_2)$$
\end{lemma}
\begin{proof}(1) is the basic property of Bergman metric, i.e. Bergman metric is invariant under bi-holomorphic map.(2) see Lemma 6 of \cite{clerc1998compressions}.
\end{proof}
We give a explicit formula of Lebesgue density $d\lambda$ on $Sym_d\CC$.  Let $e_{ij}$ denote the matrix such that all the entries are $0$ except the one in $i-$th row and $j-$th column is $1$. Let $E_{ii}=e_{ii}$ and $E_{ij}(i\neq j)=e_{ij}+e_{ji}$, then $E_{ij}, i\leq j$ form a basis of $Sym_d\CC$. 

Then we can define the Lebesgue density on $Sym_d\CC$, i.e. 
\begin{equation}\label{leb on sym}
|dE_{11}\wedge d\bar{E}_{11}\wedge\cdots\wedge dE_{ij}\wedge d\bar{E}_{ij}\wedge\cdots \wedge dE_{dd}\wedge d\bar{E}_{dd}|, i\leq j
\end{equation}

For $Z\in SD_d$, let $V(Z)d\lambda(Z)$ be a volume form on $SD_d$ induced by the Bergman metric on point $Z$. Without loss of generality, we can assume $V(0)=1$. Then we have:

\begin{lemma}\label{Berg vol}
	If $\sigma_i(Z), 1\leq i\leq d$ are the singular values of $Z$, then $$V(Z)=\Pi_{1\leq i\leq d}(1-\sigma_i(Z)^2)^{-(d+1)}$$
\end{lemma}
\begin{proof}see Lemma 6 in \cite{clerc1998compressions} for a computation of Riemann tensor for Bergman metrics for general Hermitian symmetric space.
\end{proof}
\section{Fibered rotation function and complexification of Lyapunov exponents}\label{rot mon}

Let us now fix $A\in L^\infty(X, Sp(2d,\RR))$. For $\sigma \in \RR, t\geq 0$, $\sigma+it\in \CC^+\cup \RR$, we consider the following deformation of $A$:

$$A_{\sigma+it}:= \begin{pmatrix}\cos(\sigma+it)\cdot I_d& \sin(\sigma+it)\cdot I_d\\-\sin(\sigma+it)\cdot I_d&\cos(\sigma+it)\cdot I_d\end{pmatrix}\cdot A$$

Notice that $\A_{\sigma+it}=\begin{pmatrix}e^{-t}&\\&e^t\end{pmatrix}\begin{pmatrix}e^{i\sigma}&\\&e^{-i\sigma}\end{pmatrix}\A$

The main aim of this chapter is the following theorem, which gives the complexification of $L^d(A)$:

\begin{thm}\label{rot num}
	There is a function $\zeta$ defined on $\CC^+\cup \RR$ satisfying the following properties:
	\begin{eqnarray*}
		&1.&\zeta\text{ is a holomorphic on }\CC^+.\\
		&2.&\zeta\text{'s real part } \rho \text{ is continuous on }\CC^+\cup \RR, \text{ non-increasing on } \RR.\\
		&3.&-\zeta\text{'s imaginary part }=L^d(A_{\sigma+it}), \text{ which is subharmonic on }\CC^+\cup \RR.
	\end{eqnarray*}
\end{thm}

\begin{rema}
	The function $\rho$ defined here is the \textbf{fibered rotation function } (up to multiply $2\pi$) in  \cite{avila2013monotonic}. It is a generalization of fibered rotation number for a Schr\"odinger or $SL(2,\RR)-$cocycle homotopic to identity.
	
\end{rema}

The strategy to prove Theorem \ref{rot num} is similar to the discussion in section 2 of \cite{avila2013monotonic}. But different from the $1-$dimensional case, the proof of monotonicity of the fibered rotation function is not trivial. We have to use the partial order defined in the last section.

\subsection{Definition of $\zeta$}
Denoted by $\Upsilon$ the set $$\{A\in Sp(2d,\RR), \A \cdot SPD_d\subset SPD_d\}$$where $SPD_d$ is the projective model of the disc $SD_d$.
For a matrix $A\in \Upsilon$, we can define the function $\tau_A:\overline{SD
	_d}\to GL(d,\CC)$ satisfying the following:

\begin{equation}
\A\begin{pmatrix}Z\\1\end{pmatrix}= \begin{pmatrix}\A\cdot Z\\1\end{pmatrix}\tau_{A}(Z)
\end{equation}

In fact the M\"obius transformation: $\A\cdot Z$ is well-defined for $A\in \Upsilon, Z\in \overline{SD_d}$: suppose $$\A=\begin{pmatrix}\ast&\ast\\C&D\end{pmatrix}$$ 
Since $\A\cdot \overline{SPD_d}\subset \overline{SPD_d}$, then $\begin{bmatrix}
\ast\\CZ+D
\end{bmatrix}\in \overline{SPD_d}$, therefore $CZ+D$ is invertible. As a result,  $\tau_{A}(Z)=CZ+D$.

For $\Upsilon$, denoted by $\hat{\Upsilon}$ its universal cover considered as a topological semi-group with unity $\hat{\id}$. Then there exists a unique continuous map $\hat{\tau}$:
\begin{equation}
\hat{\tau}: \hat{\Upsilon}\times \overline{SD_d}\to \CC \text{ such that }\hat{\tau}(\hat{\id},Z)=0 \text{ and }e^{i\hat{\tau}(\hat{A},Z)}=\det(\tau_A(Z))\label{tauhat}
\end{equation}
This map satisfies 
\begin{equation}
\hat{\tau}(\hat{A}_2\hat{A}_1, Z)=\hat{\tau}(\hat{A}_2, \A_1\cdot Z+\hat{\tau}(\hat{A}_1, Z))\label{gdit}
\end{equation}
and the following lemma:
\begin{lemma}\label{dpi}For any $\hat{A}\in \hat{\Upsilon}$, and any $Z,Z'\in \overline{SD_d}$,
	\begin{eqnarray}\mathfrak{I}\hat{\tau}(\hat{A},Z)=-|\ln\det(\tau_A(Z))|\label{itau}\\ 
	|\mathfrak{R}\hat{\tau}(\hat{A},Z)-\mathfrak{R}\hat{\tau}(\hat{A},Z')|<d\pi\label{arg est}
	\end{eqnarray}
\end{lemma}
\begin{proof}
	\eqref{itau} is  the consequence of \eqref{tauhat}. For \eqref{arg est}, suppose $\A=\begin{pmatrix}\ast&\ast\\C&D\end{pmatrix}$. 
	Notice that $\det(\tau_A(Z))=\det(D)\det(1+D^{-1}CZ)$, and by Proposition 2.3 of \cite{sadel2013herman},  $\norm{D^{-1}C}\leq 1$. Then by well-definedness of M\"obius transformation on $\overline{SD_d}$, the spectrum of matrix $1+D^{-1}CZ, Z\in \overline{SD_d}$ is contained in a half plane,  which implies \eqref{arg est}.
\end{proof}

Now if $\gamma: [0,1]\to \Upsilon$ is continuous, and $\hat{\gamma}:[0,1]\to \hat{\Upsilon}$ is a continuous lift, we define $\delta_\gamma\hat{\tau}(Z_0,Z_1)=\hat{\tau}(\hat{\gamma}(1), Z_1)-\hat{\tau}(\hat{\gamma}(0), Z_0)$; notice that it is independent of the choice of the lift.

For $x\in X$, consider a path $\gamma_x(s):=A_{l_{z_0,z_1}(s)}(x), s\in [0,1]$, where $l_{z_0,z_1}:[0,1]\to \CC^+\cup \RR$ is a continuous path such that $l_{z_0,z_1}(0)=z_0,l_{z_0,z_1}(1)=z_1$. Then we can define $\delta_{z_0,z_1}\xi: X\times \overline{SD_d}\times \overline{SD_d}\to \CC$ by $\delta_{z_0,z_1}\xi(x,Z_0,Z_1)=\delta_{\gamma_x}\hat{\tau}(Z_0,Z_1)$. Notice that $\delta_{z_0,z_1}\xi$ is  independent of the choice of $l_{z_0,z_1}$. 

Using the dynamics $f:X\to X$, we define 
\begin{eqnarray*}&&\delta_{z_0,z_1}\xi_n: X\times SD_d\times SD_d\to\CC\\
	&&\delta_{z_0,z_1}\xi_n(x,Z_0,Z_1):=\frac{1}{n}\sum_{k=0}^{n-1}\delta_{z_0,z_1}\xi(f^k(x), \A_{z_0}^k(x)\cdot Z_0,\A_{z_1}^k(x)\cdot Z_1 )
\end{eqnarray*}
where $\A^k(x):=\A(f^{k-1}(x))\cdots A(x)$.

We denote $\delta_z\xi$ short for $\delta_{0,z}\xi$. As in \cite{avila2013monotonic}, we study the limit of $\delta_z\xi_n$:
\begin{lemma}\label{imindep z}The limit of $\mathfrak{I}\delta_z\xi_n(x,Z_0,Z_1)$  exists for $\mu-$almost every $x$, all $z\in \CC^+$, and all $Z_0,Z_1\in SD_d$. Moreover it is independent of the choice of $Z_1$.
\end{lemma}
\begin{proof}
	The proof of $d=1$ can be found in Lemma 2.3 of \cite{avila2013monotonic}.  For general $d$, for any $Z\in SD_d$we identify $Z$ with a vector $v_1(Z)\wedge\cdots\wedge v_d(Z)\in \Lambda^d(\CC^{2d})$, where $v_i(Z)$ are the column vectors of the matrix $\begin{pmatrix}
	Z\\I_d
	\end{pmatrix}$. Therefore we define 
	\begin{eqnarray}L^d(A,x)&:=&\lim_{n\to\infty}\frac{1}{n}\ln||\Lambda^d(A)(x)||\\
	L^d(A,x, Z)&:=&\lim_{n\to\infty}\frac{1}{n}\ln||\Lambda^d(A)(x)\cdot Z||
	\end{eqnarray}
	By Oseledec theorem, the limit exists for $\mu-$almost every $x\in X$ and all $Z_0\in SD_d$. 
	
	We claim that  for $\mu-$almost every $x\in X$, for all $z\in \CC^+,Z_0,Z_1\in SD_d$,
	\begin{equation}\lim_{n\to \infty}\mathfrak{I}\delta_z\xi_n(x,Z_0,Z_1)=L^d(A, x,Z_0)-L^d(A_z, x)\label{im zeta}
	\end{equation}
	The proof is basically the same as Lemma 2.3 of \cite{avila2013monotonic}, we only need to check the following: for $z\in \CC^+ ,Z\in SD_d$, $\begin{bmatrix}Z\\I_d\end{bmatrix}$ transverses to all the Oseledec stable subspace of $\A_z(x)$.

	Consider $A_{\sigma+it}, t>0$. By Lemma \ref{schw}, $\A_{\sigma+it}$ uniformly contracts the Bergman metric of $SD_d$, and there exists a measurable function $m^+(\sigma+it, \cdot): X\to SD_d$ which is bounded from $\partial SD_d$ and holomorphically depends on $\sigma+it$, such that
	\begin{equation}
	m^+(\sigma+it, f(x))=\A_{\sigma+it}(x)\cdot m^+(\sigma+it, x)
	\end{equation}
	
	Moreover $\begin{bmatrix}m^+(\sigma+it, x)\\I_d\end{bmatrix}$ in the Grassmannian represents the unstable direction of the cocycle $\Lambda^d(\A_{\sigma+it})$ (see remark of \cite{sadel2013herman}, or section 3 and section 6 of \cite{avila2014complex}).
	
	In particular, for all $Z\in SD_d, z\in \CC^+$, the distance $$d(\A^n_z(x)\cdot Z, m^+(z, f^n(x)))$$ goes to $0$ exponentially fast, by Oseledec theorem,  $Z$ must transversal to all the Oseledec stable subspace.
\end{proof}

\begin{lemma}\label{indep z}
	For all $z\in \CC^+\cup \RR$,  $\lim_{n\to\infty}\int_X\mathfrak{R}\delta_z\xi_n(x,Z_0,Z_1)d\mu(x)$ exists for all $Z_0,Z_1\in\overline{SD_d}$. Moreover, it is continuous on $\CC^+\cup \RR$ and  independent of the choice of $Z_0,Z_1$.
\end{lemma}
\begin{proof}For $z_0,z_1\in \CC^+\cup \RR, Z_0,Z_1\in SD_d$, let $$a_n(z_0,z_1,Z_0,Z_1):=\int_X\mathfrak{R}\delta_{z_0,z_1}\xi_n(x,Z_0,Z_1)d\mu(x)$$

	As in section 2 of \cite{avila2013monotonic}, by \eqref{arg est} and \eqref{gdit}, we have 
	\begin{equation}\label{rdelta var small}
	|\mathfrak{R}\delta_{z_0,z_1}\xi_n(x,Z_0,Z_1)-\mathfrak{R}\delta_{z_0,z_1}\xi_n(x,Z_0',Z_1')|<\frac{2d\pi}{n}
	\end{equation}
	
	Then for any $n,l>0$, by $f-$invariance of $\mu$ and last equation,
	\begin{eqnarray}\label{equi an}
	&&|a_n(z_0,z_1,Z_0,Z_1)-a_l(z_0,z_1,Z_0,Z_1)|\\\nonumber
	&\leq&|a_n(z_0,z_1,Z_0,Z_1)-a_{nl}(z_0,z_1,Z_0,Z_1)|\\\nonumber
	&+&|a_l(z_0,z_1,Z_0,Z_1)-a_{nl}(z_0,z_1,Z_0,Z_1)|\\\nonumber
	&\leq&\frac{1}{l}\int_X\sum_{j=0}^{n-1}|\mathfrak{R}\delta_{z_0,z_1}\xi_n(x,Z_0,Z_1)\\\nonumber
	&-&\mathfrak{R}\delta_{z_0,z_1}\xi_n(x,A_{z_0}^{nj}(f^{-nj}(x))\cdot Z_0, A_{z_1}^{nj}(f^{-nj}(x)) \cdot Z_1)|d\mu(x)+\nonumber\\
	&+& \frac{1}{n}\int_X\sum_{j=0}^{l-1}|\mathfrak{R}\delta_{z_0,z_1}\xi_l(x,Z_0,Z_1)\nonumber\\
	&-&\mathfrak{R}\delta_{z_0,z_1}\xi_l(x,A_{z_0}^{lj}(f^{-lj}(x))\cdot Z_0, A_{z_1}^{lj}(f^{-lj}(x)) \cdot Z_1)|d\mu(x)\nonumber\\
	&\leq & \frac{2d\pi}{n}+\frac{2d\pi}{l}\nonumber
	\end{eqnarray}
	
	Since $A\in L^\infty(X,Sp(2d,\RR))$, for any $Z_0,Z_1$, $a_n(\cdot, \cdot, Z_0,Z_1)$ is continuous. By \eqref{equi an}, $a_n(\cdot, \cdot, Z_0,Z_1)$ converges ( uniformly on any bounded subset ) to a continuous function on $(\CC^+\cup \RR)^2$. By \eqref{rdelta var small} we know this limit function does not depend on the choice of $Z_0,Z_1$.
\end{proof}

\begin{rema}\label{rema mea section}By the same proof, we can prove for any measurable sections $$Z_0, Z_1: X\to \overline{SD_d}$$  $\lim_{n\to\infty}\int_X\mathfrak{R}\delta_z\xi_n(x,Z_0(x),Z_1(x))d\mu (x)$ exists and does not depend on the choice of $Z_0,Z_1$.
\end{rema}
Now we can define function $\zeta$: $$\zeta(z):=\lim_{n\to\infty}\int_X\mathfrak{R}\delta_z\xi_n(x,0,0)d\mu(x)-iL^d(A_z)$$
then we claim $\zeta$ satisfies all conditions of Theorem \ref{rot num}. 
\subsection{subharmonicity and holomorphicity} We have the following lemma for the  subharmonicity of Lyapunov exponents.

\begin{lemma}
	The map $z \mapsto L^k(A_z)$ is a subharmonic function for $A\in L^\infty(X, GL(d,\CC))$.
\end{lemma}

\begin{proof}\label{le subharmonic}
	Notice that $z\mapsto L^k(A_z)$ is the limit of the decreasing sequence of subharmonic functions $z\mapsto \frac{1}{2^n}\int_X \ln\norm{\Lambda^k(A_z)^{2^n}}_{HS}d\mu$ (see \cite{avila2011density} Lemma 2.3 for example).
\end{proof}
By \eqref{im zeta} and the subharmonicity of Lyapunov exponents we get $\zeta$ satisfies (3). of Theorem \ref{rot num}. By Lemma \ref{indep z}, $\rho$ is a continuous function on $\CC^+\cup \RR$. 

Now we prove $\zeta$ is a holomorphic function on $\CC^+$. For $z\in \CC^+$, since $m^+(z,x)$  depends on $z$ holomorphically, $$\int_X\delta_z\xi_n(x, 0, m^+(z, x))d\mu(x)$$ is a sequence of uniformly bounded holomorphic functions of $z$. By the proof of Lemma \ref{imindep z} and Remark \ref{rema mea section}, $$\lim_{n\to\infty}\int_X\delta_z\xi_n(x, 0, m^+(z, x))d\mu(x)= \lim_{n\to\infty}\int_X\delta_z\xi_n(x, 0, 0)d\mu(x).$$ 

Then by Montel theorem, $\lim_{n\to\infty}\int_X\delta_z\xi_n(x, 0, m^+(z, x))d\mu(x)$ depends on $z$ holomorphically. By the definition of $\zeta$ we get $\zeta$ is a holomorphic function on $\CC^+$.

\subsection{Fibered rotation function is non-increasing}

To prove Theorem \ref{rot num}, we only need to prove $\rho$ is non-increasing on $\RR$. At first, we give a proof for  $d=1$, which gives us the basic idea for the general case.

For all $z\in \mathbb{S}^1$, and any lift of $A\in SL(2,\RR)$, we have the following equation: 
\begin{equation}
\A\cdot z=e^{-2i\mathfrak{R}(\hat{\tau}(\hat{A}, z))}z\label{st1d}
\end{equation}

Notice that $\lim\int_X\mathfrak{R}\delta_\theta\xi_n(x,z_1,z_2)\mu(x)$ does not depend on the choice of $z_1,z_2$, we can assume $z_1,z_2\in \mathbb{S}^1$. 

Then to prove $\rho$ is non-increasing on $\RR$, by definition of $\zeta$ and \eqref{st1d} we only need to prove for any $x\in X, z\in \mathbb{S}^1, n\in \N$ and any continuous lift of the path $\A_\theta(f^n(x))\cdots \A_\theta(x)\cdot z, \theta\in\RR$, denoted as 
$$\widehat{\A_\theta(f^n(x))\cdots \A_\theta(x)\cdot z}$$ is monotonic with respect to $\theta$. Here the lift $\hat{\gamma} $ for a curve $\gamma:\RR\to \mathbb{S}^1$ is a continuous function on $\RR$ such that $\pi\circ\hat{\gamma}=\gamma$, where $\pi(x)=e^{ix}$.

In fact, for $\theta>0$ we have 
\begin{eqnarray*}&&\widehat{\A_\theta(f^n(x))\cdots \A_\theta(x)\cdot z}\\
	&=&\widehat{e^{2i\theta}\A(f^n(x))e^{2i\theta}\A(f^{n-1}(x))\cdots e^{2i\theta}\A(x)\cdot z}\\
	&>&\widehat{\A(f^n(x))e^{2i\theta}\A(f^{n-1}(x))\cdots e^{2i\theta}\A(x)\cdot z}\\
	&&\text{ (the lift of the rotation is a translation)}\\
	&>&\widehat{\A(f^n(x))\A(f^{n-1}(x))\cdots e^{2i\theta}\A(x)\cdot z}\\
	&&\text{ (the lift of the }\A\text{ action preserves the order)}\\
	&>&\cdots\\
	&>&\widehat{\A(f^n(x))\A(f^{n-1}(x))\cdots \A(x)\cdot z}
\end{eqnarray*}
Then $\rho$ is non-increasing when $d=1$.

For $d>1$, we have the following lemma to replace \eqref{st1d}, 

\begin{lemma}\label{geometric meaning of tau hat}
	For all $Z\in U_{sym}(\CC^d)$, and any lift of $A\in Sp(2d,\RR)$, 
	\begin{equation}\det(\A\cdot Z)=e^{-2i\mathfrak{R}(\hat{\tau}(\hat{A}, Z))}\det(Z)\label{stdd}
	\end{equation}
\end{lemma}
\begin{proof}
	By $\eqref{gdit}$, $\hat{\tau}$ behaves well under the iteration, so by Cartan decomposition of $Sp(2d,\RR)$, we only need to prove \eqref{stdd} for $$\A=\begin{pmatrix}U&\\&(U^{-1})^T\end{pmatrix} \text{ or }\begin{pmatrix}\frac12(S+S^{-1})&\frac12(S-S^{-1})\\\frac12(S-S^{-1})&\frac12(S+S^{-1})\end{pmatrix}$$ where $U$ is an arbitrary unitary matrix, $S$ is an arbitrary real non-singular diagonal $d\times d$ matrix.
	
	For the first case,
	\begin{equation*}\det(\A\cdot Z)=\det(UZU^T)=\det(U)^2\det(Z)=e^{-2i\mathfrak{R}(\hat{\tau}(\hat{A}, Z))}\det(Z)
	\end{equation*}
	
	For the second case,
	\begin{eqnarray*}
		&&\det(\A\cdot Z)\\
		&=&\det((S+S^{-1})Z+(S-S^{-1}))\det((S-S^{-1})Z+(S+S^{-1}))^{-1}\\
		&=&\det((S+S^{-1})+(S-S^{-1})\overline{Z})\det((S-S^{-1})Z+(S+S^{-1}))^{-1}\\
		&&\cdot \det(Z) (\text{ since }Z\in U_{sym}(\CC^d), Z^{-1}=\overline{Z})\\
		&=&e^{-2i\text{Arg}(\det((S-S^{-1})Z+(S+S^{-1})))}\det(Z)(\text{ since }S \text{ is a real matrix})\\
		&=&e^{-2i\mathfrak{R}(\hat{\tau}(\hat{A}, Z))}\det(Z)
	\end{eqnarray*}
\end{proof}

Come back to the proof of the non-increasing property of $\rho$. As in the case $d=1$, by \eqref{stdd}, we have to prove for all $x\in X, Z\in U_{sym}(\CC^d)$, any continuous lift of the path $\det(\A_\theta(f^n(x))\cdots \A_\theta(x)\cdot Z), \theta\in\RR$ is monotonic with respect to $\theta$.

In other words, we need to prove for any continuous lift of the path $\A_\theta(f^n(x))\cdots \A_\theta(x)\cdot Z$, denoted as $\widehat{\A_\theta(f^n(x))\cdots \A_\theta(x)\cdot Z}$, we have that $\hat{\det}(\widehat{\A_\theta(f^n(x))\cdots \A_\theta(x)\cdot Z})$ is monotonic with respect to $\theta$, where $\hat{\det}$ is defined in \eqref{def hatdet} .

By (3), (4) of Lemma \ref{order on covering}, using the order defined in the subsection \ref{partial order}, as the $1-$dimensional case, we have for $\theta>0$,
\begin{equation}\widehat{\A_\theta(f^n(x))\cdots \A_\theta(x)\cdot Z}>\widehat{\A(f^n(x))\cdots \A(x)\cdot Z}
\end{equation}

But by (1). of Lemma \ref{order on covering}, $\hat{\det}$ is monotonic with respect to the order $"<"$. Combining with last equation, we have that the function $\hat{\det}(\widehat{\A_\theta(f^n(x))\cdots \A_\theta(x)\cdot Z})$ is monotonic with respect to $\theta$, which completes the proof of Theorem \ref{rot num}.

\section{A Kotani theoretic estimate}\label{key chapter }
The main aim of this chapter is Theorem \ref{m+=m-}, which is a higher dimensional generalization of Lemma 2.6 in \cite{avila2013monotonic}. We introduce the concept of $m^--$function firstly.  

\subsection{The $m^-$-function}
By Lemma \ref{schw}, $\A_{\sigma-it}^{-1}, t>0$ contracts the Bergman metric uniformly on $SD_d$.  we can define $m^-(\sigma-it, \cdot)\in SD_d, t>0$ which depends on $\sigma-it$ holomorphically, such that 
\begin{equation}\label{m-def}
m^-(\sigma-it, f(x))=\A_{\sigma-it}(x)\cdot m^-(\sigma-it, x)
\end{equation}

For later use, we consider the following property of $m^-$: for $t>0$, by \eqref{m-def} and the definition of function $\tau_{(\cdot)}(\cdot)$, there exists a function $\tau_{A_{\sigma-it}(x)}(m^-(\sigma-it,x))\in GL(d,\CC)$ such that 
\begin{equation}\label{tau_-def}
\A_{\sigma-it} 
\begin{pmatrix}m^-(\sigma-it, x)\\I_d
\end{pmatrix}=
\begin{pmatrix}m^-(\sigma-it,f(x))\\ I_d 
\end{pmatrix} \tau_{A_{\sigma-it}(x)}(m^-(\sigma-it,x))
\end{equation}
Moreover we have:
\begin{lemma}\label{conj m-}
	\begin{equation}
	\A_{\sigma+it}\begin{pmatrix}I_d\\\overline{m^-(\sigma-it,x)}
	\end{pmatrix}=
	\begin{pmatrix}I_d\\
	\overline{m^-(\sigma-it, f(x))} 
	\end{pmatrix}\overline{\tau_{A_{\sigma-it}(x)}(m^-(\sigma-it,x))}
	\end{equation}
\end{lemma}
\begin{proof}We denote $A$ for $A_{\sigma+it}$, $A_-$ for $A_{\sigma-it}$, $m^{-}$ for $m^{-}(\sigma- it, x)$, $\tm^{-}$ for $m^{-}(\sigma- it, f(x))$, $\tau_-$ for $\tau_{A_{\sigma-it}(x)}(m^-(\sigma-it,x))$. Recall that $C$ is the Cayley element defined in \eqref{Cayley ele}, then by \eqref{tau_-def} we have 
	\begin{eqnarray*}
		\A_-\begin{pmatrix}m_-\\I_d\end{pmatrix}&=&\begin{pmatrix}\tm_-\\I_d\end{pmatrix}\tau_-\\
		CA_-C^{-1}\begin{pmatrix}m_-\\I_d\end{pmatrix}&=&\begin{pmatrix}\tm_-\\I_d\end{pmatrix}\tau_-(\text{ by definition of }\A)\\
		A_-C^{-1}\begin{pmatrix}m_-\\I_d\end{pmatrix}&=&C^{-1}\begin{pmatrix}\tm_-\\I_d\end{pmatrix}\tau_-\\
	\end{eqnarray*}
	Take complex conjugate for both sides of last equation, we have 
	\begin{eqnarray*}
		A\overline{C^{-1}}\begin{pmatrix}\overline{m_-}\\I_d\end{pmatrix}&=&\overline{C^{-1}}\begin{pmatrix}\overline{\tm_-}\\I_d\end{pmatrix}\overline{\tau_-}\\
		AC^{-1}(C\overline{C^{-1}}\begin{pmatrix}\overline{m_-}\\I_d\end{pmatrix})&=&\overline{C^{-1}}\begin{pmatrix}\overline{\tm_-}\\I_d\end{pmatrix}\overline{\tau_-}\\
		CAC^{-1}(C\overline{C^{-1}}\begin{pmatrix}\overline{m_-}\\I_d\end{pmatrix})&=&C\overline{C^{-1}}\begin{pmatrix}\overline{\tm_-}\\I_d\end{pmatrix}\overline{\tau_-}\\
		\A(C\overline{C^{-1}}\begin{pmatrix}\overline{m_-}\\I_d\end{pmatrix})&=&C\overline{C^{-1}}\begin{pmatrix}\overline{\tm_-}\\I_d\end{pmatrix}\overline{\tau_-}\\
	\end{eqnarray*}
	Notice that $C\overline{C^{-1}}=\begin{pmatrix}0&I_d\\I_d&0\end{pmatrix}$, we have 
	\begin{equation*}
	\A\begin{pmatrix}I_d\\ \overline{m^-}\end{pmatrix}=\begin{pmatrix}I_d\\ \overline{\tm^-}\end{pmatrix}\overline{\tau_-}
	\end{equation*}
\end{proof}

Now we can state Theorem \ref{m+=m-}. 
\begin{thm}\label{m+=m-}
	For almost every $\sigma_0\in \RR$ such that $L(A_{\sigma_0})=0$, we have that:
	
	(1).	\begin{eqnarray*}
		\limsup_{t\to 0^+}&&\int_X\frac{1}{1-\norm{m^+(\sigma_0+it, x)}^2} d\mu(x)\\+&&\int_X \frac{1}{1-\norm{m^-(\sigma_0-it, x)}^2}d\mu(x)<\infty
	\end{eqnarray*}
	
	(2).	\begin{eqnarray*}
		\liminf_{t\to 0^+}\int_X \norm{m^+(\sigma_0+it, x))-m^-(\sigma_0-it,x)}^2d\mu(x)=0
	\end{eqnarray*}
	
\end{thm}

To prove Theorem \ref{m+=m-}, we introduce the following concepts.

\subsection{$q-$function and Lyapunov exponents}
\begin{definition}\label{q def}
	Consider the derivative of the holomophic map $Z\mapsto \A_{\sigma+it}(x)\cdot Z$ at point $m^+(\sigma+it,x)$, denote $q_{\sigma+it}(x )$ as the Jacobian of the derivative with respect to the volume form induced by the Bergman metric.
\end{definition}
By the discussion before Lemma \ref{Berg vol}, we have the following expression of $q-$function:
\begin{lemma}\label{qcomp}
	\begin{equation}q_{\sigma+it}(x)=|\frac{dm^+(\sigma+it, f(x))}{dm^+(\sigma+it, x)}|\frac{V(m^+(\sigma+it, f(x))}{V(m^+(\sigma+it, x))}
	\end{equation} 
	where $|\frac{dm^+(\sigma+it, f(x))}{dm^+(\sigma+it, x)}|$ is the Jacobian of the map $Z\mapsto \A_{\sigma+it}(x)\cdot Z$ at point $m^+(\sigma+it,x)$ with respect to the Lebesgue measure $d\lambda$ defined in \eqref{leb on sym}.
\end{lemma}
The following lemma gives a explicit formula of $|\frac{dm^+(\sigma+it, f(x))}{dm^+(\sigma+it, x)}|$.
\begin{lemma}\label{jac leb} $|\frac{dm^+(\sigma+it, f(x))}{dm^+(\sigma+it, x)}|=|\det(\tau_{A_{\sigma+it}}(x))|^{-2(d+1)}$
\end{lemma}

\begin{proof}
	We only need to prove the following: for an arbitrary $Z\in Sym_d\CC, \begin{pmatrix}A&B\\C&D
	\end{pmatrix}\in Sp(2d,\CC)$ such that $CZ+D$ is invertible, the map 
	\begin{eqnarray*}
		Sym_d\CC&\to& Sym_d\CC\\
		X&\mapsto& (AX+B)(CX+D)^{-1}
	\end{eqnarray*} at point $Z$ has Jacobian (with respect to $d\lambda$) $|\det(CZ+D)|^{-2(d+1)}$.
	
	At first, by computation the tangent map of $X\to (AX+B)(CX+D)^{-1}$ at point $Z$ is: 
	\begin{eqnarray*}
		Sym_d\CC&\mapsto& Sym_d\CC\\
		H&\to&  (A-(AZ+B)(CZ+D)^{-1}C)\cdot H\cdot (CZ+D)^{-1}
	\end{eqnarray*}
	We need the following equation for symplectic group:
	\begin{lemma}\label{idty sp}
		For an arbitrary $Z\in Sym_d\CC, \begin{pmatrix}A&B\\C&D
		\end{pmatrix}\in Sp(2d,\CC)$ such that $CZ+D$ is invertible: we have that 
		\begin{equation}
		(A-(AZ+B)(CZ+D)^{-1}C)=(CZ+D)^{{-1}^{T}}
		\end{equation}
	\end{lemma}
	\begin{proof}
		Since $\begin{pmatrix}A&B\\C&D\end{pmatrix}\in Sp(2d,\CC)$, we have that
		\begin{equation}\label{prop sp}
		A^TC,B^TD\text{ are symmetric. } A^TD-C^TB=1
		\end{equation}
		Moreover, for all $Z\in Sym_d\CC$ such that $CZ+D$ is invertible, we have that 
		\begin{equation}\label{prop sp2}
		(AZ+B)(CZ+D)^{-1}\text{ is symmetric.}
		\end{equation}
		then
		\begin{equation}
		(AZ+B)(CZ+D)^{-1}=(D^T+ZC^T)^{-1}(B^T+ZA^T)
		\end{equation}
		By \eqref{prop sp2}, to prove Lemma \ref{idty sp}, we have to prove:
		\begin{equation}
		(A-(D^T+ZC^T)^{-1}(B^T+ZA^T)C)(CZ+D)^T=I_d
		\end{equation}
		Multiply by $D^T+ZC^T$ from the left to both sides, we need to prove
		\begin{equation}
		(D^T+ZC^T)A(CZ+D)^T=(B^T+ZA^T)C(CZ+D)^T+(D^T+ZC^T)
		\end{equation}which is the consequence of \eqref{prop sp}.
	\end{proof}
	
	Come back to the proof of Lemma \ref{jac leb}, by last lemma the tangent map of $X\to (AX+B)(CX+D)^{-1}$ at point $Z$ is 
	\begin{eqnarray*}
		Sym_d\CC&\mapsto& Sym_d\CC\\
		H&\to& (CZ+D)^{{-1}^{T}}\cdot H\cdot (CZ+D)^{-1}
	\end{eqnarray*}
	So Lemma \ref{jac leb} is the consequence of the following lemma:
	\begin{lemma}
		Suppose $g\in GL(d,\CC)$, the linear map 
		\begin{eqnarray*}
			Sym_d\CC&\to& Sym_d\CC\\
			H&\mapsto& g^TH g
		\end{eqnarray*}
		has jacobian $|\det g|^{2(d+1)}$ with respect to the density $d\lambda$ on $Sym_d\CC$.
	\end{lemma}
	\begin{proof}
		The Jacobian behaves well under the multiplication on $GL(d,\CC)$. By the polar decomposition of $GL(d,\CC)$, we only need to prove the lemma in the case $g$ is diagonal or $g$ is contained in the unitary group.
		When $g$ is diagonal, the lemma can be verified by computation directly. Notice the Jacobian of the map gives a homomorphism from $GL(d,\CC)$ to $(\RR^+,\times)$. So it maps the unitary group to the unique compact subgroup of $(\RR^+,\times)$: the identity. \end{proof}
	
\end{proof}

By our construction of $m^+$, for $t>0$, $\begin{bmatrix}m^+\\I_d\end{bmatrix}$ represents the unstable direction of the cocycle $\Lambda^d(\A)$. As a result, we have

\begin{equation}\label{ldtau}
L^d (A_{\sigma+it})=\int_X \ln|\det\tau_{A_{\sigma+it}(x)}(m^+(\sigma+it,x))|d\mu(x)
\end{equation}

Combining \eqref{ldtau} and Lemma \ref{qcomp}, \ref{jac leb}, we get 
\begin{equation}\label{ldq}
L^d(A_{\sigma+it})=\frac{1}{2(d+1)}\int_X-\ln q_{\sigma+it}(x)d\mu(x)
\end{equation}

\subsection{Boundary behavior of Lyapunov exponents}
As in Kotani theory and \cite{avila2013monotonic}, using the results of Theorem \ref{rot num}, we get the following lemma for boundary behavior of Lyapunov exponents.

\begin{lemma}\label{kotanimono}
	For almost every $\sigma_0\in \RR$ such  that $L(A_{\sigma_0})=0$, 
	\begin{equation}\label{l/t=l'}
	\lim_{t\to 0^+}\frac{L^d(A_{\sigma_0+it})}{t}-\frac{\partial L^d(A_{\sigma_0+it})}{\partial t}=0
	\end{equation}
\end{lemma}
\begin{proof}
	We follow the proof of Theorem 2.5 in \cite{avila2013monotonic}. By upper semi-continuity of $L^d$, for every $\sigma_0\in \RR$ such  that $L(A_{\sigma_0})=0$, we have 
	\begin{equation}
	\lim_{t\to 0^+}L^d(A_{\sigma_0+it})=0
	\end{equation}
	Then
	\begin{eqnarray*}
		\lim_{t\to 0^+}\frac{L^d(A_{\sigma_0+it})}{t}&=&\lim_{t\to 0^+}\frac{L^d(A_{\sigma_0+it})-L^d(A_{\sigma_0+i0^+})}{t}\\
		&=&\lim_{t\to 0^+}\frac{\int_{0+}^t\frac{\partial L^d(A_{\sigma_0+it})}{\partial t}dt}{t}
	\end{eqnarray*}
	To prove Lemma \ref{kotanimono}, we only need to prove the following limit exists for almost every $\sigma_0\in \RR$.
	\begin{equation}\label{partialld}
	\lim_{t\to 0^+}\frac{\partial L^d(A_{\sigma_0+it})}{\partial t}
	\end{equation}
	
	By Cauchy-Riemann equations,
	\begin{equation}
	\frac{\partial L^d(A_{\sigma_0+it})}{\partial t}=-\frac{\partial \rho}{\partial \sigma}(\sigma_0+it)
	\end{equation}
	By Theorem \ref{rot num}, since the map $\rho$ is harmonic on $\CC^+$, continuous on $\CC^+\cup \RR$, non-increasing on $\RR$, one can say that for Lebesgue almost every $\sigma_0\in \RR$, (see Theorem 2.5 of \cite{avila2013monotonic})
	\begin{equation}
	\lim_{t\to 0}\frac{\partial \rho}{\partial \sigma}(\sigma_0+it)=\frac{d}{d\sigma}\rho(\sigma_0)
	\end{equation}
	Since $\rho$ is non-increasing, the derivative of $\rho$ on $\RR$ exists almost every where, which implies the limit in \eqref{partialld} exists for almost every $\sigma_0$.
\end{proof}
\subsection{Proof of Theorem \ref{m+=m-}}
Now we come back to the proof of theorem. By Lemma \ref{kotanimono}, for almost every $\sigma_0\in \RR$ such that $L(A_{\sigma_0})=0$, we have \eqref{l/t=l'} holds, $\lim_{t\to 0^+}\frac{\partial L^d(A_{\sigma_0+it})}{\partial t}$ exists and is finite. 

We claim that for these $\sigma_0$, equations (1).(2).of Theorem \ref{m+=m-} hold. From now to the end of the proof of Theorem \ref{m+=m-}, we denote for simplicity $m^\pm$ for $m^\pm(\sigma_0\pm it, x)$, $\tilde{m}^\pm$ for $m^\pm(\sigma_0\pm it,f(x))$, $\tau$ for $\tau_{A_{\sigma_0+it}(x)}(m^+(\sigma_0+it,x))$, $\tau_-$ for $\tau_{A_{\sigma_0-it}(x)}(m^-(\sigma-it,x))$, $A$ for $A_{\sigma_0+it}(x)$, $A_-$ for $A_{\sigma_0-it}(x)$, $L^d$ for $L^d(A_{\sigma_0+it})$, $q$ for $q_{\sigma_0+it}(x)$ .

\subsubsection{proof of (1). of Theorem \ref{m+=m-}}
Notice that $\A_{\sigma_0+it}=\begin{pmatrix}e^{-t}&\\&e^{t}\end{pmatrix}\A_{\sigma_0}$, we have an expression of $q$ by the singular values of $\tilde{m}$.
\begin{lemma}\label{q compu}
	\begin{equation}q^{-1}=e^{-2t(d^2+d)}\cdot\Pi_ {i=1}^d(\frac{e^{4t}(1-\sigma_i(\tilde{m}^+)^2)}{1-e^{4t}\sigma_i(\tilde{m}^+)^2})^{d+1}
	\end{equation}
\end{lemma}
\begin{proof}
	By Lemma \ref{qcomp} and the definition of $q$,
	\begin{eqnarray*}
		q^{-1}&=&\frac{V(m^+)}{V(\tm^+)}\cdot |\frac{dm^+}{d\tm^+}|\\
		&=&\frac{V(e^{2t}\tm^+)}{V(\tm^+)}\cdot \frac{V(m^+)}{V(e^{2t}\tm^+)}|\frac{dm^+}{de^{2t}\tm^+}|\cdot  e^{2t(d^2+d)}\\&&(\text{ since }SD_d\text{ has }d^2+d \text{ real dimension})\\
		&=&\frac{V(e^{2t}\tm^+)}{V(\tm^+)}\cdot e^{2t(d^2+d)}\\&&(\text{ since }m\mapsto e^{2t}\tm^+\text{ is an isometry for Bergman metric })\\
		&=&e^{-2t(d^2+d)}\cdot\Pi_{i=1}^d(\frac{e^{4t}(1-\sigma_i(\tilde{m})^2)}{1-e^{4t}\sigma_i(\tilde{m})^2})^{d+1}(\text{ by Lemma }\ref{Berg vol})
	\end{eqnarray*}
\end{proof}
Using that for $r>0, 0\leq s<e^{-r}$ we have 
\begin{equation}\label{ele ln ineq}\ln(\frac{e^r(1-s)}{1-e^rs})\geq \frac{r}{1-s}
\end{equation}
by last lemma, we get
\begin{eqnarray*}\ln q^{-1}&\geq& -2t(d^2+d)+\sum_{i=1}^d(d+1)\cdot\frac{4t}{1-\sigma_i(\tm^+)^2}\\&=&2(d+1)t\sum_{i=1}^d\frac{1+\sigma_i(\tm^+)^2}{1-\sigma_i(\tm^+)^2}
\end{eqnarray*}
By \eqref{ldq}, since $L^d=\frac{1}{2(d+1)}\int_X\ln q^{-1}d\mu$, we have 
\begin{equation}
L^d\geq t\int_X\sum_{i=1}^d\frac{1+\sigma_i(\tm^+)^2}{1-\sigma_i(\tm^+)^2}d\mu
\end{equation}
An analogous argument yields
\begin{equation}
L^d\geq t\int_X\sum_{i=1}^d\frac{1+\sigma_i(\tm^-)^2}{1-\sigma_i(\tm^-)^2}d\mu
\end{equation}

Then we conclude that
\begin{equation}\label{l/t}
\frac{L^d}{t}\geq \frac{1}{2}\int_X\sum_{i=1}^d(\frac{1+\sigma_i(\tm^+)^2}{1-\sigma_i(\tm^+)^2}+\frac{1+\sigma_i(\tm^-)^2}{1-\sigma_i(\tm^-)^2})d\mu
\end{equation}
By our assumption of $\sigma_0$, we get the proof of (1).

\subsubsection{map $\Lambda$ and basis $\mathfrak{B}(\cdot, \cdot)$ }
To prove (2)., we consider the following map:
\begin{definition}
	Let $Mat_{2d, d}(\CC)$ be the space of all $2d\times d$ complex matrices, we can define the map:
	\begin{eqnarray*}
		\Lambda: Mat_{2d, d}(\CC)&\to& \Lambda^d(\CC^{2d})\\
		X&\mapsto& x_1\wedge\cdots\cdots\wedge x_d
	\end{eqnarray*}
	where $\{x_i, 1\leq i\leq d\}$ are the column vectors of $X$.
\end{definition}
The following lemma lists some properties of $\Lambda$ we will use later. Recall that for $A\in GL(2d,\CC)$,  $\Lambda^k(A)$ is the natural action induced by $A$ on $\Lambda^k(\CC^{2d})$. For arbitrary two $2d\times d$ matrices $X,Y$, denote 
\begin{eqnarray}
D\Lambda(X)(Y):=\lim_{t\to 0}\frac{\Lambda(X+tY)-\Lambda(X)}{t}
\end{eqnarray}
\begin{lemma}\label{Lambdaprop}
	For $A\in GL(2d,\CC),B\in GL(d,\CC), X,Y\in Mat_{2d, d}(\CC)$, supppose that $X=(x_1,\dots, x_d), Y=(y_1,\dots,y_d)$, where $\{x_i, 1\leq i\leq d\}, \{y_i, 1\leq i\leq d\}$ are the column vectors of $X,Y$ respectively, then we have the following equations:
	\begin{eqnarray}
	\Lambda^d(A)\cdot\Lambda(X)&=&\Lambda(A\cdot X)\\
	D\Lambda(X)(Y)&=&\sum_{i=1}^d x_1\wedge\cdots\wedge x_{i-1}\wedge y_i\wedge x_{i+1}\wedge\cdots\wedge x_d\label{DLambdaXY}\\
	\text{           }D\Lambda(AX)(AY)&=&\Lambda^d(A)\cdot D\Lambda(X)(Y)\label{DLambdachangebase}\\
	\Lambda(X\cdot B)&=&\det(B)\Lambda(X)\label{rightmult}
	\end{eqnarray}
	
\end{lemma}

\begin{proof}By computation directly.
\end{proof}

From now on we identify $\Lambda^{2d}(\CC^{2d})$ with $\CC$ as the following:\\
\textbf{Identification}
If $\varpi\in \Lambda^{2d}(\CC^{2d})=c(\varpi)\cdot e_1\wedge\dots\wedge e_{2d}$, then we identify $\varpi$ with $c(\varpi)$. Here $e_i$ are standard basis of $\CC^{2d}$.

Now we define a collection of basis of $\Lambda^d(\CC^{2d})$ for later use.
\begin{definition}
	Suppose $X,Y\in Mat_{2d,d}(\CC)$ are with rank $d$, and the column vectors $\{x_i, 1\leq i\leq d\}, \{y_i, 1\leq i\leq d\}$ of $X,Y$ are linearly independent, then the following subset in $\Lambda^d(\CC^{2d})$ forms a basis of $\Lambda^d(\CC^{2d})$:
	\begin{eqnarray*}\{x_{i_1}\wedge\cdots\wedge x_{i_{|I|}}\wedge y_{j_1}\wedge\cdots\wedge y_{j_{|J|}}: I,J\subset \{1,\dots, d\},\\
		|I|+|J|=d,i_1<i_2<\dots, j_1<j_2<\dots\}
	\end{eqnarray*}
	denoted by $\mathfrak{B}(X,Y)$.
	
	For any element $\omega\in\Lambda^d(\CC^{2d})$, the coefficient of $x_1\wedge\cdots\wedge x_d$ for the expansion of $\omega$ with respect to the basis $\mathfrak{B}(X,Y)$ is 
	\begin{equation}\label{wedgecoef}
	\frac{\omega\wedge(y_1\wedge\cdots\wedge y_d)}{x_1\wedge\cdots\wedge x_d\wedge y_1\wedge\cdots\wedge y_d}
	\end{equation}
	Here we use the identification above.
\end{definition}

We will use the following lemma later.
\begin{lemma}\label{d lambda expansion on B}Suppose $X,Z\in Mat_{2d,d}(\CC)$ are with rank $d$, and the column vectors $\{x_i, 1\leq i\leq d\}, \{z_i, 1\leq i\leq d\}$ of $X,Z$ are linearly independent, then the coefficient of $\Lambda(X)$ for the expansion of $D\Lambda(X)(Y)$ with respect to the basis $\mathfrak{B}(X,Z)$ is the trace of the matrix $\begin{pmatrix}X,Z
	\end{pmatrix}^{-1}\cdot Y$. (the trace of a $2d\times d$ matrix is the sum of the diagonal entries)
\end{lemma}
\begin{proof}Suppose $\begin{pmatrix}X,Z
	\end{pmatrix}^{-1}\cdot Y=\begin{pmatrix}a_{ij}\\b_{ij}
	\end{pmatrix}_{1\leq i\leq d, 1\leq j\leq d}$, then 
	\begin{equation}
	Y=\begin{pmatrix}\dots&\sum_{k=1}^d (a_{ki}x_k+b_{ki}z_k)&\dots
	\end{pmatrix}_{1\leq i\leq d}
	\end{equation}
	By \eqref{DLambdaXY} we get that 
	\begin{eqnarray*}
		D\Lambda(X)(Y)&=&\sum_{i=1}^d x_1\wedge\cdots\wedge x_{i-1}\wedge y_i\wedge x_{i+1}\wedge\cdots\wedge x_d\\
		&=&\sum_{i=1}^d x_1\wedge\cdots\wedge x_{i-1}\wedge (\sum_{k=1}^d(a_{ki}x_k+b_{ki}z_k))\wedge\cdots\wedge x_d\\
		&=&\sum_{i=1}^d a_{ii}x_1\wedge\cdots\wedge x_d+\text{ other term in }\mathfrak{B}(X,Z)\\
		&=&(\text{trace of }\begin{pmatrix}X,Z
		\end{pmatrix}^{-1}\cdot Y) \Lambda(X)+\\&&\text{ other terms in }\mathfrak{B}(X,Z)
	\end{eqnarray*}which gives the proof.\end{proof}
\subsubsection{proof of (2). of Theorem \ref{m+=m-}}
Come back to the proof of (2). At first, 
\begin{equation}\label{beftauinv}\A\begin{pmatrix}m^+\\ I_d\end{pmatrix}=\begin{pmatrix}\tm^+\\I_d\end{pmatrix}\tau
\end{equation}
Take the inverse,

\begin{equation}\label{tauinv}
\A^{-1}\begin{pmatrix}\tm^+\\I_d\end{pmatrix}=\begin{pmatrix}m^+\\ I_d\end{pmatrix}\tau^{-1}
\end{equation}

Let the operator $\Lambda$ acting on both sides of \eqref{tauinv}, we get:
\begin{eqnarray}
\Lambda(\A^{-1}\begin{pmatrix}\tm^+\\I_d\end{pmatrix})&=&\Lambda(\begin{pmatrix}m^+\\ I_d\end{pmatrix}\tau^{-1})
\end{eqnarray}
Then differentiate with respect to $t$, 
\begin{eqnarray}\label{tauinv'}
\frac{\partial}{\partial t}\Lambda(\A^{-1}\begin{pmatrix}\tm^+\\I_d\end{pmatrix})&=&\frac{\partial}{\partial t}(\frac{1}{\det\tau}\Lambda(\begin{pmatrix}m^+\\ I_d\end{pmatrix})(\text{ by }\eqref{rightmult})
\end{eqnarray}
Using Lemma \ref{Lambdaprop} to compute the derivative,  we get
\begin{eqnarray*}
	\text{left of }\eqref{tauinv'} &=& D\Lambda(\A^{-1}\begin{pmatrix}\tm^+\\I_d
	\end{pmatrix})(\frac{\partial}{\partial t} (\A^{-1}\begin{pmatrix}\tm^+\\I_d
\end{pmatrix}))\\
&=&D\Lambda(\A^{-1}\begin{pmatrix}\tm^+\\I_d
\end{pmatrix})(-\A^{-1}(\frac{\partial}{\partial t}\A)\A^{-1}\begin{pmatrix}\tm^+\\I_d
\end{pmatrix}\\
&&+\A^{-1}\begin{pmatrix}I_d\\0
\end{pmatrix}\frac{\partial \tm^+}{\partial t}
)\\
&=&-\Lambda^d(\A^{-1})\cdot D\Lambda(\begin{pmatrix}\tm^+\\I_d\end{pmatrix})(\begin{pmatrix}-\tm^+\\I_d\end{pmatrix}-\begin{pmatrix}I_d\\0
\end{pmatrix}\frac{\partial \tm^+}{\partial t}
)
\end{eqnarray*}
where we use \eqref{DLambdachangebase} and $\frac{\partial}{\partial t}\A=\begin{pmatrix}-I_d&\\&I_d
\end{pmatrix}\A$ in the last equality.
\begin{eqnarray*}\text{right of }\eqref{tauinv'}&=&-\frac{1}{(\det\tau)^2}\frac{\partial \det\tau}{\partial t}\Lambda(\begin{pmatrix}m^+\\ I_d\end{pmatrix})\\
	&&+\frac{1}{\det\tau}D\Lambda(\begin{pmatrix}m^+\\I_d\end{pmatrix})(\begin{pmatrix}I_d\\0\end{pmatrix}\frac{\partial m^+}{\partial t})
\end{eqnarray*}
Notice that
\begin{eqnarray*}\Lambda^d(\A)\cdot \Lambda(\begin{pmatrix}m^+\\ I_d\end{pmatrix})&=&\Lambda(\A\begin{pmatrix}m^+\\ I_d\end{pmatrix})\\
	&=&\Lambda(\begin{pmatrix}\tm^+\\ I_d\end{pmatrix}\tau)\\
	&=&\det\tau\cdot \Lambda(\begin{pmatrix}\tm^+\\ I_d\end{pmatrix})
\end{eqnarray*}
Applying $-\Lambda^d(\A)$ to both sides of \eqref{tauinv'}, by previous discussion we have\\
\textbf{the key equation}
\begin{eqnarray*}
	D\Lambda(\begin{pmatrix}\tm^+\\I_d\end{pmatrix})(\begin{pmatrix}-\tm^+\\I_d\end{pmatrix})&-&D\Lambda(\begin{pmatrix}\tm^+\\I_d\end{pmatrix})(\begin{pmatrix}I_d\\0
	\end{pmatrix}\frac{\partial \tm^+}{\partial t}
	)=\\
	\frac{1}{\det\tau}\frac{\partial \det\tau}{\partial t}\Lambda(\begin{pmatrix}\tm^+\\ I_d\end{pmatrix})&-&
	\frac{1}{\det\tau}\Lambda^d(\A)\cdot D\Lambda(\begin{pmatrix}m^+\\I_d\end{pmatrix})(\begin{pmatrix}I_d\\0\end{pmatrix}\frac{\partial m^+}{\partial t})
\end{eqnarray*}

\subsubsection{the key equation}
To analyse each term of the key equation, for $\begin{pmatrix}\tm^+\\I_d
\end{pmatrix}, \begin{pmatrix}I_d\\\overline{\tm^-}
\end{pmatrix}\in Mat_{2d,d}(\CC)$, we consider the basis 
$\mathfrak{B}(\begin{pmatrix}\tm^+\\I_d \end{pmatrix}, 
\begin{pmatrix} 
I_d \\ \overline{\tm^-}\end{pmatrix})$. This is actually a basis since $\norm{\tm^+},\norm{\overline{\tm^-}}<1$, 
$\det\begin{pmatrix}\tm^+& I_d\\I_d&\overline{\tm^-}
\end{pmatrix}\neq 0$.

For the key equation, the following lemmas give the coefficients of $\Lambda(\begin{pmatrix}\tm^+\\I_d \end{pmatrix} )
$ for the expansion of each term with respect to the basis $\mathfrak{B}(\begin{pmatrix}\tm^+\\I_d \end{pmatrix}, 
\begin{pmatrix} 
I_d \\ \overline{\tm^-}\end{pmatrix})$.
\begin{lemma}\label{coef1} The coeffcient of $\Lambda(\begin{pmatrix}\tm^+\\I_d \end{pmatrix})$ for the expansion of $D\Lambda(\begin{pmatrix}\tm^+\\I_d\end{pmatrix})(\begin{pmatrix}-\tm^+\\I_d\end{pmatrix})$ with respect to the basis $\mathfrak{B}(\begin{pmatrix}\tm^+\\I_d \end{pmatrix}, 
	\begin{pmatrix} 
	I_d \\ \overline{\tm^-}\end{pmatrix})$ is the trace of $(I_d-\overline{\tm^-}\tm^+)^{-1}(I_d+\overline{\tm^-}\tm^+)$.
\end{lemma}
\begin{proof}
	
	By Lemma \ref{d lambda expansion on B},  to prove lemma \ref{coef1}, we only need to compute \begin{equation}\begin{pmatrix}\tm^+&I_d\\I_d&\overline{\tm^-}
	\end{pmatrix}^{-1} \cdot \begin{pmatrix}-\tm^+\\I_d
	\end{pmatrix}
	\end{equation}
	In fact $\begin{pmatrix}\tm^+&I_d\\I_d&\overline{\tm^-}
	\end{pmatrix}^{-1}=\begin{pmatrix}(\overline{\tm^-}\tm^+ -I_d)^{-1}
	\overline{\tm^-}& (I_d-\overline{\tm^-}\tm^+)^{-1}
	\\I_d+\tm^+(I_d-\overline{\tm^-}\tm^+)^{-1}
	\overline{\tm^-}&-\tm^+(I_d-\overline{\tm^-}\tm^+)^{-1}
	\end{pmatrix}
	$
	Then we get 
	\begin{equation}\begin{pmatrix}\tm^+&I_d\\I_d&\overline{\tm^-}
	\end{pmatrix}^{-1} \cdot \begin{pmatrix}-\tm^+\\I_d
	\end{pmatrix}=\begin{pmatrix}(I_d-\overline{\tm^-}\tm^+)^{-1}(I_d+\overline{\tm^-}\tm^+)\\ \ast
	\end{pmatrix}
	\end{equation}
	By Lemma \ref{d lambda expansion on B}, we get the proof of Lemma \ref{coef1}.
\end{proof}
\begin{lemma}\label{coef2}
	The coeffcient of $\Lambda(\begin{pmatrix}\tm^+\\I_d \end{pmatrix})$ for the expansion of $D\Lambda(\begin{pmatrix}\tm^+\\I_d\end{pmatrix})(\begin{pmatrix}I_d\\0
	\end{pmatrix}\frac{\partial \tm^+}{\partial t}
	)$ with respect to the basis $\mathfrak{B}(\begin{pmatrix}\tm^+\\I_d \end{pmatrix}, 
	\begin{pmatrix} 
	I_d \\ \overline{\tm^-}\end{pmatrix})$ is $$\det\begin{pmatrix}\tm^+&I_d\\I_d&\overline{\tm^-}
	\end{pmatrix}^{-1}\cdot D\Lambda(\begin{pmatrix}\tm^+\\I_d\end{pmatrix})(\begin{pmatrix}I_d\\0
	\end{pmatrix}\frac{\partial \tm^+}{\partial t}
	)\wedge \Lambda(\begin{pmatrix} 
	I_d \\ \overline{\tm^-}\end{pmatrix})$$
\end{lemma}
\begin{proof}Using \eqref{wedgecoef}.
\end{proof}
\begin{lemma}\label{coef3}
	The coeffcient of $\Lambda(\begin{pmatrix}\tm^+\\I_d \end{pmatrix})$ for the expansion of $$\frac{1}{\det\tau}\Lambda^d(\A)\cdot D\Lambda(\begin{pmatrix}m^+\\I_d\end{pmatrix})(\begin{pmatrix}I_d\\0\end{pmatrix}\frac{\partial m^+}{\partial t})$$ with respect to the basis $\mathfrak{B}(\begin{pmatrix}\tm^+\\I_d \end{pmatrix}, 
	\begin{pmatrix} 
	I_d \\ \overline{\tm^-}\end{pmatrix})$ is 
	$$\det\begin{pmatrix}m^+&I_d\\I_d&\overline{m^-}
	\end{pmatrix}^{-1}\cdot D\Lambda(\begin{pmatrix}m^+\\I_d\end{pmatrix})(\begin{pmatrix}I_d\\0
	\end{pmatrix}\frac{\partial m^+}{\partial t}
	)\wedge \Lambda(\begin{pmatrix} 
	I_d \\ \overline{m^-}\end{pmatrix})$$
\end{lemma}
\begin{proof}Let $X=\begin{pmatrix}m^+\\I_d\end{pmatrix}, W_1=\begin{pmatrix}I_d\\0\end{pmatrix}\frac{\partial m^+}{\partial t},W_2=\begin{pmatrix}I_d\\\overline{m^-}\end{pmatrix}$
	By \eqref{wedgecoef} we get
	\begin{eqnarray*}&&\text{the coefficient }\\
		&=&\frac{1}{\det\tau}(\Lambda^d(\A)\cdot D\Lambda(X)(W_1))\wedge \Lambda(\begin{pmatrix} 
			I_d \\ \overline{\tm^-}\end{pmatrix})\cdot \det\begin{pmatrix}\tm^+&I_d\\I_d&\overline{\tm^-}
		\end{pmatrix}^{-1}\\
		&=&\frac{1}{\det\tau\cdot\det\overline{\tau_-}}(\Lambda^d(\A)\cdot D\Lambda(X)(W_1))\wedge (\Lambda^d(\A)\cdot \Lambda(W_2))\\
		&&\cdot \det\begin{pmatrix}\tm^+&I_d\\I_d&\overline{\tm^-}
		\end{pmatrix}^{-1}
		\text{ }(\text{use Lemma }\ref{conj m-})\\
		&=&\frac{1}{\det\tau\cdot\det\overline{\tau_-}}\Lambda^{2d}(\A)\cdot (D\Lambda(X)(W_1)\wedge \Lambda(W_2))\cdot \det\begin{pmatrix}\tm^+&I_d\\I_d&\overline{\tm^-}
		\end{pmatrix}^{-1}
	\end{eqnarray*}
	To prove Lemma \ref{coef3}, we only need to prove the following equation:
	\begin{equation}
	\det(\A)\det\begin{pmatrix}m^+&I_d\\I_d&\overline{m^-}\end{pmatrix}=\det\tau\det\overline{\tau_-}\det\begin{pmatrix}\tm^+&I_d\\I_d&\overline{\tm^-}
	\end{pmatrix}
	\end{equation}
	which is just a corollary of \eqref{beftauinv} and Lemma \ref{conj m-}.
\end{proof}
Now come back to the key equation.  By Lemma \ref{coef1},\ref{coef2},\ref{coef3}, taking the coefficient of $\Lambda(\begin{pmatrix}\tm^+\\I_d \end{pmatrix})$ in the key equation and integrating with respect to the measure $\mu$, we have
\begin{equation}
\int_X\text{tr}((I_d-\overline{\tm^-}\tm^+)^{-1}(I_d+\overline{\tm^-}\tm^+))d\mu=\int_X\frac{1}{\det\tau}\frac{\partial \det\tau}{\partial t}d\mu
\end{equation}
Consider the real part, which gives
\begin{equation}\int_X \mathfrak{R}(\text{tr}((I_d-\overline{\tm^-}\tm^+)^{-1}(I_d+\overline{\tm^-}\tm^+)))d\mu=\frac{\partial L^d}{\partial t}
\end{equation}

\subsubsection{A trace inequality and the rest of the proof}
By \eqref{l/t} and Lemma \ref{kotanimono}, we have that 
\begin{eqnarray*}
	&&\liminf_{t\to 0^+}\int_X\frac{1}{2}\sum_{i=1}^d(\frac{1+\sigma_i(\tm^+)^2}{1-\sigma_i(\tm^+)^2}+\frac{1+\sigma_i(\tm^-)^2}{1-\sigma_i(\tm^-)^2})\\
	&&-\mathfrak{R}(\text{tr}((I_d-\overline{\tm^-}\tm^+)^{-1}(I_d+\overline{\tm^-}\tm^+)))d\mu\\
	&\leq&\lim_{t\to 0^+}\frac{L^d(A_{\sigma_0+it})}{t}-\frac{\partial L^d(A_{\sigma_0+it})}{\partial t}\\&=&0
\end{eqnarray*}
Compare with (2). of Theorem \ref{m+=m-}, to finish the proof, we only need to prove the following inequality:
\begin{lemma}\label{finalstep}
	\begin{eqnarray*}&&
		\frac{1}{2}\sum_{i=1}^d(\frac{1+\sigma_i(\tm^+)^2}{1-\sigma_i(\tm^+)^2}+\frac{1+\sigma_i(\tm^-)^2}{1-\sigma_i(\tm^-)^2})\\&-&\mathfrak{R}(\text{tr}((I_d-\overline{\tm^-}\tm^+)^{-1}(I_d+\overline{\tm^-}\tm^+)))\\
		&\geq& \norm{\tm^+-\tm^-}^2_{HS}
	\end{eqnarray*}
\end{lemma}
\begin{proof}
	Notice that for $\norm{x}<1, \frac{1+x}{1-x}=2(1-x)^{-1}-1=-1+2\sum_{k=0}^\infty x^k$, and $\overline{\tm^-}=(\tm^-)^\ast$ We have that:
	\begin{eqnarray*}
		&&\text{left of Lemma } \ref{finalstep}\\
		&=&\sum_{i=1}^d\frac{1}{1-\sigma_i(\tm^+)^2}+\frac{1}{1-\sigma_i(\tm^-)^2}-2\mathfrak{R}\text{tr}((I_d-\overline{\tm^-}\tm^+)^{-1}\\
		&=&\sum_{k=0}^\infty(\sum_{i=1}^d\sigma_i(\tm^+)^{2k}+\sigma_i(\tm^-)^{2k})-2\mathfrak{R}\text{tr}(((\tm^-)^\ast\tm^+)^k)\\
		&=&\sum_{k=0}^\infty\text{tr}(((\tm^+)^\ast\tm^+)^k)+\text{tr}(((\tm^-)^\ast\tm^-)^k)-2\mathfrak{R}\text{tr}(((\tm^-)^\ast\tm^+)^k)
	\end{eqnarray*}
	Then the proof of Lemma \ref{finalstep} is the consequence of the following matrix inequalities: for arbitrary $d\times d$ complex matrices $X,Y, k>1$, 
	\begin{eqnarray*}
		\text{tr}((X^\ast X)^k+(Y^\ast Y)^k)&\geq& 2\mathfrak{R}\text{tr}((X^\ast Y)^k)\\
		\text{tr}(X^\ast X)+\text{tr}(Y^\ast Y)-2\mathfrak{R}\text{tr}(X^\ast Y)&=&\norm{X-Y}^2_{HS}
	\end{eqnarray*} \end{proof}

	\section{Density of positive Lyapunov exponents for continuous symplectic cocycle}\label{density cont}
	\subsection{Herglotz function}
	We recall that  $m$ is called a Herglotz (matrix valued) function if $m$ is an analytic matrix valued function defined on $\CC^+$ and $Im(m(z))$ is a positive definite Hermitian matrix for all $z\in \CC^+$, we list some basic properties we will use (see \cite{gesztesy2000matrix}).
	\begin{lemma}\label{herglotz}
		The function $m(\cdot)$ has a finite normal limit $m(\sigma+i0^+)=\lim_{t\to 0^+}m(\sigma+it)$ for a.e. $\sigma\in \RR$. Moreover if two Herglotz function $m_1,m_2$ have the same limit on a positive measure set on $\RR$, then $m_1=m_2$.
		
	\end{lemma}
	Notice that $\Phi_C^{-1}\cdot m^+(\cdot, x),\Phi_C^{-1}\cdot m^-(-\cdot, x)$ are Herglotz functions.
	\subsection{$M-$function}
	Consider the following definition of $M-$function, which is introduced in \cite{avila2005generic}. 
	\begin{definition}For $A\in L^{\infty}(X,Sp(2d,\RR))$, we denote 
		\begin{equation}
		M(A):=\text{ the Lebesgue measure of }\{\theta\subset [0,2\pi], L(A_\theta)=0\}
		\end{equation}
	\end{definition}
	We hope to prove for generic  $A$, $M(A)=0$. At first, we prove it for a family of symplectic cocycles taking finitely many values. 
	\subsection{Symplectic cocycles taking finitely many values}
	We introduce the following definition of \textit{deterministic}, which is similar to the definition for Sch\"odinger operator in \cite{simon1983kotani} and \cite{kotani1989jacobi}.
	\begin{definition}
		For $A\in L^{\infty}(X,Sp(2d,\RR))$, we say $A$ is deterministic if $A(f^n(x)),n\geq 0$ is a.e., a measurable function of $\{A(f^n(x)), n<0\}$.
	\end{definition}
	
	As \cite{kotani1989jacobi}, we have the following theorem for the $M-$function for symplectic cocycles. 
	\begin{thm}\label{finitekotanitheory}Suppose $A\in L^{\infty}(X,Sp(2d,\RR))$ such that
		
		(1).$A(x), x\in X$ only takes finitely many values.
		
		(2).$A(f^n(x)), n\in \ZZ$, is not periodic for almost every $x\in X$.
		
		(3).If $A(x)\neq A(y), x,y\in X$, then $\A(x)^{-1}(0)\neq \A(y)^{-1}(0)$.
		
		We have  $M(A)=0$.
	\end{thm}
	
	\begin{proof}We know that for almost every $x\in X, A(f^n(x)), n\geq 0$ can determine the function $m^-(x)$. In fact, for $z$ such that $\mathfrak{I}(z)<0, \A_z(x)^{-1}$ uniformly contracts the Bergman metric on $SD_d$, so like the property of $m-$function in Kotani theory, we have that 
		\begin{equation}\label{mfunctioncomp}
		m^-(z, x)=\lim_{n\to \infty}\A_z(x)^{-1}\cdots\A_z(f^n(x))^{-1}\cdot 0
		\end{equation}
		But we also have the following lemma for the inverse problem:
		\begin{lemma}\label{inverseproblem}
			If a cocycle $A\in L^{\infty}(X,Sp(2d,\RR))$ satisfies (1),
			(3) of Theorem \ref{finitekotanitheory}, then the function $m^-(z, \cdot), z\in \CC^+$ determines $\{A(f^n(\cdot)),n\geq 0\}$ in the sense that if $x,y\in X$ such that $A(f^n(x)), A(f^n(y)), n\geq 0$ are bounded, and $m^-(\cdot,x)=m^-(\cdot, y)$, then $A(f^n(x))=A(f^n(y)), n\geq 0$.
		\end{lemma}
		\begin{proof}
			Let $z$ tends to $\infty$ along the line $\{\mathfrak{R}(z)=0,\mathfrak{I}(z)<0\}$ in \eqref{mfunctioncomp}, we get 
			\begin{equation}\lim_{\mathfrak{R}(z)=0, \mathfrak{I}(z)\to -\infty}m^-(z,x)=\A(x)^{-1}(0)
			\end{equation}
			By (3) of Theorem \ref{finitekotanitheory}, we know that $m^-(\cdot, x)$ can determine $\A(x)$, by 
			\begin{equation}\A_z(x)\cdot m^-(z,x)=m^-(z,f(x))
			\end{equation}
			it implies $m^-(\cdot, x)$ can determine $m^-(\cdot, f(x))$, using the same method again, we can determine $\A(f(x))$. Repeat this process, we determine all $A(f^n(x)),n\geq 0$.
		\end{proof}
		Come back to the proof of Theorem \ref{finitekotanitheory}. Suppose $M(A)>0$, we claim that under the assumptions (1),(3), $A$ must be deterministic. Then by Kotani's argument in \cite{kotani1989jacobi}, $A$ must be periodic, which contradicts the assumption (2).
		
		In fact, the set $\{A(f^n(x)), n<0\}$ determines  $m^+(\cdot, x)$. If $M(A)>0$, by (2). of Theorem \ref{m+=m-},  $m^+(\cdot, x)$ determines $m^-(\cdot, x)$ on a full mesure subset of $\{\theta: L(A_\theta)=0\}$. 
		
		By Lemma \ref{herglotz}, since $\Phi_C^{-1}\cdot m^+(\cdot, x), \Phi_C^{-1}\cdot m^-(-\cdot, x)$ are Herglotz functions,
		$m^+(\cdot,x)$ determines $m^-(-\cdot, x)$ on all of $\CC^+$. By Lemma \ref{inverseproblem}, $\{A(f^n(x)), n\geq 0\}$ is determined by $\{A(f^n(x)), n<0\}$. That means $A$ is deterministic.
	\end{proof}

	\subsection{Continuous symplectic cocycles}\label{semicontinuity arg}
	\begin{thm}\label{generic+}
		Suppose $f$ is not periodic on $supp(\mu)$, then the set of $A$ such that $L(A)>0$ is dense in $C(X, Sp(2d,\RR))$. 
	\end{thm}
	\begin{proof}
		At first we consider the following lemma:
		\begin{lemma}\label{generic+erg}
			Suppose $f:(X,\mu)\to (X,\mu)$ ($f\neq\id$) is ergodic, then there is a residual subset of cocycles $A$ in $C(X,Sp(2d,\RR))$ such that $M(A)=0$.
		\end{lemma}	
		\begin{proof}
			We follow the proof in \cite{avila2005generic}. At first we consider the following lemma:
			\begin{lemma}\label{zdense}
				There exists a dense subset $\mathcal{Z}$ of $L^{\infty}(X,Sp(2d,\RR))$ satisfying all conditions of Theorem \ref{finitekotanitheory}.
			\end{lemma}
			
			\begin{proof}By Lemma 2 of \cite{avila2005generic}, the cocycles in $L^{\infty}(X,Sp(2d,\RR))$ satisfying the first two conditions of Theorem \ref{finitekotanitheory} are dense in $L^{\infty}(X,Sp(2d,\RR))$. But for each cocycle $A$ satisfying the first two condition of Theorem \ref{finitekotanitheory}, we can find a new cocycle $A'$ satisfying all conditions in Theorem \ref{finitekotanitheory} and arbitrarily close to $A$.
			\end{proof}
			\begin{lemma}\label{upp cont}For every $r>0$, the map
				\begin{eqnarray*}(L^1(X, Sp(2d,\RR))\cap B_r(L^\infty(X,Sp(2d,\RR))), \norm{\cdot}_1)&\to& \RR,\\
					A &\mapsto& M(A)
				\end{eqnarray*}is upper semi-continuous.
			\end{lemma}
			\begin{proof}
				The proof is the same as the $SL(2,\RR)$ case, since we have the formula in \cite{sadel2013herman} to replace the Herman-Avila-Bochi formula in \cite{avila2002formula} for $SL(2,\RR)$ case. And by Theorem \ref{rot num}, $L^d(A_z)$ is harmonic for $z\in \CC^+$ and subharmonic on $\CC^+\cup\RR$, we can move the proof for $SL(2,\RR)$ case in \cite{avila2005generic} to here.
			\end{proof}

			\begin{lemma}\label{msmalldense}
				For $A\in C(X,Sp(2d,\RR)), \epsilon>0, \delta>0$, there is an $A'\in C(X,Sp(2d,\RR))$ such that $\norm{A-A'}_\infty<\epsilon, M(A)<\delta$.
			\end{lemma}
			\begin{proof}The proof is almost the same as Lemma 3 of \cite{avila2005generic}, we only need to use the set $\mathcal{Z}$ in Lemma \ref{zdense} and Theorem \ref{finitekotanitheory} to replace the set $\mathcal{Z}$ and Kotani result in Lemma 3 of \cite{avila2005generic}.
			\end{proof}
			
			Come back to the proof of Lemma \ref{generic+erg}, for $\delta>0$, we define 
			$$M_\delta=\{A\in C(X,Sp(2d,\RR): M(A)<\delta\}$$
			By Lemma \ref{upp cont}, $M_\delta$ is open, and by Lemma \ref{msmalldense}, $M_\delta$ is dense. It follows that 
			$$\{A\in C(X,Sp(2d,\RR):M(A)=0\}=\cap_{\delta>0}M_{\delta}$$is residual.
		\end{proof}
		Come back to the proof of Theorem \ref{generic+}, Let $P \subset X$ be the set of periodic orbits of $f$. If $\mu(P) < 1$, then using Lemma \ref{generic+erg} we get the proof.
		
		Assume $\mu(P)=1$, we follow the argument of Lemma 3.1 in \cite{avila2011density}. Let $P_k \subset X$ be the set of periodic orbits of period $k \geq 1$. Since $f$ is not periodic on $supp(\mu)$, $\mathcal{P}_n =\cup_{k\leq n}P_k\neq supp(\mu)$ for every $n \geq 1$. Thus there are arbitrarily large $n$ such that $\mu(\mathcal{P}_n\backslash \mathcal{P}_{n-1}) > 0$. 
		
		Choose such a large $n$, and take $x\in supp(\mu)|\mathcal{P}_n\backslash \mathcal{P}_{n-1}$. We can approximate any $A\in C(X, Sp(2d,\RR))$ by some $A'$ which is constant 
		in a compact neighborhood $K$ of $\{f^k(x)\}^{n-1}
		_{k=0}$. The details of the following argument can be found in the Appendix \ref{app A}, here we only give an outline. We will prove that there is a constant $C$ independent of $n, f$  such that for generic $\{A'(f^k(x))\}$, there exist $\theta\in (-\frac{C}{n}, \frac{C}{n})$ with $L(A_\theta'(x))>0$. (Since $A'$ is locally constant near the orbit of $x\in supp(\mu)$, we have $L(A_\theta')>0$. )
		
		Otherwise there is an open interval $I$ contained $0$ and  $|I|>O(\frac{1}{n})$, such that for all $\theta\in I$, all the eigenvalues of $A_\theta'^{(n)}(x))$ are norm $1$, where $|I|$ is the length of $I$.
		But in the Appendix \ref{app A} we will prove for any interval $I'$ such that 
		\begin{equation}\label{elliptic}
		\forall \theta\in I', \text{ all the eigenvalues of }A_\theta'^{(n)}(x)) \text{ are simple and norm 1}
		\end{equation}
		we have $$|I'|\leq O(\frac{1}{n})$$As a result, there are two intervals $I_1, I_2\subset I$ satisfying \eqref{elliptic} and sharing common boundary point $\theta_0$ such that $A_{\theta_0}'^{(n)}(x)$ has repeated eigenvalues with norm $1$. But this can not happen for generic choice of $\{A'(f^k(x))\}$.
		
	\end{proof}

	\section{The proof of Theorem \ref{mainresult}}\label{complexification arg}
	By Theorem \ref{generic+}, as in the $SL(2,\RR)-$case, to prove Theorem \ref{mainresult}, we need a local regularization formula similar to Theorem 7 in \cite{avila2011density}.
	
	At first we need the following lemma:
	\begin{lemma}\label{proper image implie harmon}Suppose $A\in C(X,Sp(2d,\RR))$, $\Omega\subset \CC$ is a domain. Suppose an analytic $Sp(2d,\CC)$-valued map $B$ is defined on $\Omega$ such that for all $z\in \Omega, x\in X$, $\overset{\circ}{B(z)}\cdot \overline{SD_d}\subset SD_d$. Then the Lyapunov exponent $L^d(B(z)A)$ harmonically depends on $z\in \Omega$.
	\end{lemma}
	\begin{proof}The proof is basically the same as the discussion in Chapter \ref{rot mon} for holomorphicity of $\zeta-$function. See the remark at page 7 of \cite{sadel2013herman}, and section 3 and 6 of \cite{avila2014complex}. 
	\end{proof}

	As in \cite{avila2011density}, let $\norm{\cdot}_\ast$ denote the sup norm in the space $C(X, \mathfrak{sp}(2d,\RR))$ and $C(X, \mathfrak{sp}(2d, \CC))$. And for $r>0$,  let $\mathcal{B}_\ast(r), \mathcal{B}^\CC_\ast(r)$ be the corresponding $r-$ ball. For $A\in C(X, Sp(2d,\RR))$, $a,b\in C(X, \mathfrak{sp}(2d,\RR))$, we define the following function 
	\begin{equation}
	\Phi_{\epsilon}(A,a,b):=\int_{-1}^1\frac{1-t^2}{|t^2+2it+1|^2}L^d(e^{\epsilon(tb+(1-t^2)a)}A)dt
	\end{equation}
	The following local regularization formula is the main result of this chapter.
	\begin{thm}\label{deformsynp}
		There exists $\eta>0$ such that if $b\in C(X, \mathfrak{sp}(2d,\RR))$ is $\eta-$close to $\begin{pmatrix}0&I_d\\-I_d&0\end{pmatrix}$, then for every $\epsilon>0$, and every $A\in C(X, Sp(2d,\RR))$,
		\begin{equation}\label{poissoncontract}
		\overset{\circ}{e^{\epsilon(zb+(1-z^2)a)}}\cdot \overline{SD_d}\subset {SD_d}
		\end{equation}
		when 
		
		(1).$z\in \{|z|=1\}\cap \mathfrak{I}(z)>0 \text{ or } z=(\sqrt{2}-1)i, a\in \mathcal{B}^\CC_\ast(\eta)$,
		
		(2).$z\in\{|z|<1\}\cap \mathfrak{I}(z)>0,  a\in \mathcal{B}_\ast(\eta)$.

		Moreover 
		\begin{equation}\label{poissoncont}
		a\mapsto \Phi_\epsilon(A,a,b)
		\end{equation}
		is a continuous function of $a\in \mathcal{B}_\ast(\eta)$ and depends continuously (as an analytic function) on $A$.
	\end{thm}
	\begin{proof}In fact we only need to prove \eqref{poissoncontract}, \eqref{poissoncont} is the consequence of \eqref{poissoncontract} and Lemma \ref{proper image implie harmon}, see Theorem 7 of \cite{avila2011density}.
		
		To prove \eqref{poissoncontract}, we claim there exists a positive number $\eta>0$ such that for every point $Z\in\partial SD_d$, $\{Z^T=Z, \norm{Z}=1\}$, for $\epsilon>0$ small, the path $Z_\epsilon:=\overset{\circ}{e^{\epsilon(zb+(1-z^2)a)}}\cdot Z$ is contained in $SD_d$ for $z$ and $a$ in either case (1) or (2). This implies there exists $\epsilon_0>0$ small, for all $\epsilon<\epsilon_0$, $\overset{\circ}{e^{\epsilon(zb+(1-z^2)a)}}\cdot \overline{SD_d}\subset SD_d$. By iteration, $\overset{\circ}{e^{\epsilon(zb+(1-z^2)a)}}$ takes $ \overline{SD_d}$ into $SD_d$ for every $\epsilon>0$. 
		
		At first, by the "left-oriented" Zassenhaus formula (see \cite{casas2012efficient}, for example), we have the following equation for exponential map of matrix when $\epsilon$ is small, $\norm{X},\norm{Y}\leq 2$.
		\begin{equation}\label{zass}
		e^{\epsilon(X+Y)}=e^{O(\epsilon^2\norm{X}\cdot \norm{Y})}e^{\epsilon X}e^{\epsilon Y}
		\end{equation}
		which means there exist a vector $W$ in the Lie algebra with norm less than $O(\epsilon^2\norm{X}\cdot \norm{Y})$, such that $e^{\epsilon(X+Y)}=e^We^{\epsilon X}e^{\epsilon Y}$.
		
		In addition, we need some notations for a real Lie algebra $\mathfrak{g}$ and its complexification $\mathfrak{g}^\CC=\mathfrak{g}\oplus i\mathfrak{g}$.
		For an element $c\in \mathfrak{g}^\CC, a,b\in \mathfrak{g} $ such that $c=a+ib$, we denote \begin{equation}\mathfrak{R}(c)=a, \mathfrak{I}(c)=b
		\end{equation}
		
		From now to the end of this chapter, we always consider $\mathfrak{g}$ is the Lie algebra of $U(d,d)\cap Sp(2d,\CC)$ or $\RR$. Then $\mathfrak{g}^\CC$ is $\mathfrak{sp}(2d,\CC)$ or $\CC$.
		
		Now we denote  $R(a,b,z)=\mathfrak{R}(z\overset{\circ}{b}+(1-z^2)\overset{\circ}{a})=\mathfrak{R}(z)\overset{\circ}{b}+\mathfrak{R}((1-z^2)\overset{\circ}{a})$ and $I(a,b,z)=\mathfrak{I}(z\overset{\circ}{b}+(1-z^2)\overset{\circ}{a})=\mathfrak{I}(z)\overset{\circ}{b}+\mathfrak{I}((1-z^2)\overset{\circ}{a})$.
		
		Let $\eta$ be small, then for $z,a$ in either case (1) or (2) we have the following equations:
		\begin{eqnarray}
		Z_\epsilon=\overset{\circ}{e^{\epsilon(zb+(1-z^2)a)}}\cdot Z&=&e^{\epsilon (R+iI)}\cdot Z\\
		e^{\epsilon R}\cdot Z&\in& \partial{SD_d}\label{realpresbdry}\\
		I(a,b,z)&=&\mathfrak{I}(z)(\begin{pmatrix}i&\\&-i\end{pmatrix}+O(\eta))\label{keyof2011}\\
		\norm{R(a,b,z)}_\ast&\leq& 2\label{r<2}\\
		\norm{I(a,b,z)}_\ast&\leq&2\mathfrak{I}(z)\label{i<2}
		\end{eqnarray}
		Here \eqref{keyof2011} comes from the fact that  $\norm{\mathfrak{I}((1-z^2)\overset{\circ}{a})}_\ast\leq O(\eta\mathfrak{I}(z))$ holds for either case (1) or (2).
		
		Denote $Z'=e^{\epsilon R}Z$. Then by \eqref{realpresbdry} we know $\norm{Z'}=1$, and we have:
		\begin{eqnarray*}
			Z_\epsilon&=&e^{\epsilon (R+iI)}\cdot Z\\
			&=&e^{O(\epsilon^2\norm{R}_\ast\norm{I}_\ast)}e^{\epsilon iI}e^{\epsilon R}\cdot Z\text{ by }\eqref{zass}\eqref{r<2}\eqref{i<2}\\
			&=&e^{O(\epsilon^2\mathfrak{I}(z))}e^{\epsilon iI}\cdot Z' \text{ by }\eqref{r<2}\eqref{i<2}\\
			&=&e^{O(\epsilon^2\mathfrak{I}(z))}e^{\epsilon \mathfrak{I}(z)(\begin{pmatrix}-1&\\&1\end{pmatrix}+O(\eta))}\cdot Z' \text{ by }\eqref{keyof2011}\\
			&=&e^{O(\epsilon^2\mathfrak{I}(z))}e^{O(\epsilon^2\eta \mathfrak{I}(z)^2)}e^{O(\epsilon\eta\mathfrak{I}(z))}\cdot (e^{-2\epsilon \mathfrak{I}(z)}Z') \text{ by }\eqref{zass}\\
			&=&e^{O(\epsilon(\epsilon+\eta)\mathfrak{I}(z))}\cdot(e^{-2\epsilon \mathfrak{I}(z)}Z') \text{ since }\eta \text{ is small. }\\
		\end{eqnarray*}
		If $\epsilon$ is small enough, since the action on the equation above is a M\"obius transformation, then by computation we have 
		\begin{eqnarray*}
			\norm{Z_\epsilon}&=& \norm{e^{O(\epsilon(\epsilon+\eta)\mathfrak{I}(z))}\cdot (e^{-2\epsilon \mathfrak{I}(z)}Z')}\\
			&\leq &e^{-\epsilon \mathfrak{I}(z)}
		\end{eqnarray*}for $\epsilon$ small, 
		which implies $Z_\epsilon\in SD_d$. Then we get the proof of Theorem \ref{deformsynp}.
	\end{proof}

	To finish the proof of Theorem \ref{mainresult} we need the following short lemma.
	\begin{lemma}\label{shortlemma}
		Let $A\in C(X,Sp(2d,\RR))$, $a,b\in C(X,\frak{sp}(2d,\RR))$ and $\epsilon>0$, if $L^d(e^{\epsilon a})>0$, then $\Phi_\epsilon(A,a,b)>0$.
	\end{lemma}
	
	\begin{proof}
		As in the proof of Lemma \ref{le subharmonic} and Lemma 2.3 in \cite{avila2011density}, the map $$\gamma: t\mapsto L^d(e^{\epsilon(tb+(1-t^2)a)}A)$$ is a subharmonic function. 
		
		By subharmonicity, if $\gamma(t)=0$ for almost every $t\in (-1,1)$, then $\gamma(t)=0$ for all $t\in (-1,1)$. As a result, if $\gamma(0)>0$, $\Phi_\epsilon(A,a,b)$ must be positive. 
	\end{proof}
	
	Now we can prove Theorem \ref{mainresult}. We follow the argument in \cite{avila2011density}. For all $\delta>0$, $A\in \mathfrak{B}\subset C(X,Sp(2d,\RR))$, where $\mathfrak{B}$ is ample, we need to prove there is a $v\in \mathfrak{b}\subset C(X,\mathfrak{sp}(2d,\RR))$ such that $\norm{v}_{\mathfrak{b}} <\delta, L^d(e^vA)>0$.
	
	Choose a  positive number $\eta$ satisfying the conditions in Theorem \ref{deformsynp}, and take $b\in \mathfrak{b}$ such that $b$ is $\eta-$close to $\begin{pmatrix}
	0&I_d\\
	-I_d&0
	\end{pmatrix}$, let $\epsilon>0$ such that $\epsilon \norm{b}_{\mathfrak{b}}<\frac{\delta}{2}$. By Theorem \ref{generic+}, there is an element $a\in B_\ast(\eta)\subset C(X,\mathfrak{sp}(2d,\RR))$ such that $L^d(e^{\epsilon a}A)>0$. By Lemma \ref{shortlemma}, we have
	$\Phi_\epsilon(A,a,b)>0$
	
	Since $\mathfrak{b}$ is dense in $C(X,\mathfrak{sp}(2d,\RR))$, and by Theorem \ref{deformsynp} we know the map in \eqref{poissoncont} is continuous,  there is an element $a'\in B_\ast(\eta)\cap \mathfrak{b}$ such that $\Phi_\epsilon(A,a',b)>0$.
	
	By Theorem \ref{deformsynp}, the map $\gamma': s\mapsto \Phi_\epsilon(A,sa',b)$ is an analytic function of $s\in [-1,1]$. Since $\gamma'(1)>0$, we can choose $$0<s< \min\{1, \frac{\delta}{2\epsilon \norm{a'}_{\mathfrak{b}}}\}$$ such that $\gamma'(s)>0$. Then there exists $t\in (-1,1)$ such that $$L^d(e^{\epsilon(tb+(1-t^2)sa')}A)>0$$
	Let $v=\epsilon(tb+(1-t^2)sa')$, then $v\in \mathfrak{b}$,  $\norm{v}_{\mathfrak{b}}<\epsilon(\norm{b}_{\mathfrak{b}}+s\norm{a'}_{\mathfrak{b}})<\delta$
	and $L^d(e^vA)>0$.

	\section{proof of the rest of results}
	
	\subsection{Proof of Corollary \ref{main2}: $SHSp(2d)$ and $SU(d,d)$ cocycles}
	The proof of Corollary \ref{main2} for $SHSp(2d)$ and $SU(d,d)$ cocycles is similar to Theorem \ref{mainresult}. Consider the following correspondences between the concepts used in symplectic cocycle and special Hermitian symplectic cocycle. Almost all the following concept also works for any biholomorphic transformation group for (non compact) Hermitian symmetric space.
	\begin{eqnarray*}
		Sp(2d,\RR)&\leftrightarrow& SHSp(2d)\\
		U(d,d)\cap Sp(2d,\CC)&\leftrightarrow& SU(d,d)\\
		SH_d&\leftrightarrow&\{Z=X+iY,\\
		&& X,Y\in Her(d), Y>0\}\\
		SD_d&\leftrightarrow&\{Z, I_d-Z^\ast Z>0\} \\
		\text{Caylay element}\frac{1}{\sqrt{2}} \begin{pmatrix}I_d&-i\cdot I_d\\I_d&i\cdot I_d \end{pmatrix}&\leftrightarrow&\frac{1}{\sqrt{2}} \begin{pmatrix}I_d&-i\cdot I_d\\I_d&i\cdot I_d \end{pmatrix}\\
		A\mapsto \A:=CAC^{-1}&\leftrightarrow&A\mapsto \A:=CAC^{-1}, \\
		Sp(2d,\RR)\cong U(d,d)\cap Sp(2d,\CC)&& SHSp(2d)\cong SU(d,d)\\
		\text{M\"obius transformation}&\leftrightarrow&\text{M\"obius transformation}\\
		\Phi_C: SH_d\to SD_d&\leftrightarrow& \Phi_C: \{Z=X+iY, X,Y\in Her(d), \\
		&& Y>0\} \to \{Z, I_d-Z^\ast Z>0\}\\
		\partial SD_d &\leftrightarrow& \partial \{Z, I_d-Z^\ast Z>0\}=\{\norm{Z}=1\}\\
		\partial_k SD_d &\leftrightarrow&\{\norm{Z}=1, \text{rank}(1-Z^\ast Z)=d-k\}\\
		\partial_dSD_d=U_{sym}(\CC^d)&\leftrightarrow& U(d)\\
		\text{fin}\partial_dSH_d=Sym_d\RR&\leftrightarrow&her(d)\\
		\Phi_C\text{ gives chart }: Sym_d\RR &\leftrightarrow&\Phi_C:her(d)\\
		\to \{\det(Z-1)\neq 0\}\cap \partial_dSD_d&~~~~~&\to \{Z,\norm{Z}=1, \det(Z-1)\neq 0\}\\
	\end{eqnarray*}
	\begin{eqnarray*}
	\text{atlas }\Phi_{C_{g_k}}&\leftrightarrow&\Phi_{C_{g_k}} \text{ be defined similarly.}\\
		\text{Bergman metric, volume form}&\leftrightarrow&\text{Bergman metric, volume form}\\&& \text{on } \{\norm{Z}< 1\}\\
		V(Z)&\leftrightarrow&V(Z)=\Pi_{1\leq i\leq d}(1-\sigma_i(Z)^2)^{-2d}\\
	\end{eqnarray*}
	Consider $A_{\sigma+it},\A_{\sigma+it}$ defined as in the symplectic case, we hope to prove the same result as Theorem \ref{rot num} firstly. Obviously we can define $\tau_A$, $\hat{\tau}$ as in the  symplectic case, and we also have similar results to Lemma \ref{dpi}, \ref{indep z}. 
	
	By \cite{clerc1998compressions}, Lemma \ref{schw} holds for all (non compact) bounded Hermitian symmetric space, then we can define $m^+, m^-$ as in the symplectic case. So Lemma \ref{imindep z} also works for special pseudo unitary group. So we can define function $\zeta$ which satisfies condition 1 and 3 in Theorem \ref{rot num}. 
	
	To prove statement in Theorem \ref{rot num} for $SHSp(2d)$, we need to prove the fibered rotation function, the real part of $\zeta$ is non-increasing on $\RR$. 
	
	Define the cone field $\{h>0\}$ on $her(d)$, and then use the tangent map of $\Phi_{C_{g_k}}$ to map it to $TU(d)$. As in the  symplectic case, it gives well-defined cone fields $\mathcal{C},\hat{\mathcal{C}}$ on $U(d)$ and $\widehat{U(d)}$. Using $\hat{\mathcal{C}}$ we give a partial order on $\widehat{U(d)}$ as the $\widehat{U_{sym}(\CC^d)}$ case. 
	
	By Cartan decomposition of $SHSp(2d)$ and identity $$\det(1+XY)=\det(1+YX)$$we can get the same equation as in Lemma \ref{geometric meaning of tau hat}. Then we can prove the non-increasing property of the fibered rotation function as Theorem \ref{rot num}.
	
	To prove the same result as Theorem \ref{m+=m-}. We can use the following equation to replace Lemma \ref{conj m-}.

	Suppose $t>0$, $m^-=m^-(\sigma-it, x), \tm^-=m^-(\sigma-it,f(x))$ satisfying $\tm^-=\A_{\sigma-it}(x)\cdot m^-$. There exists $\tau^-\in GL(d,\CC)$ such that
	\begin{equation}\label{mstar inv}
	\A\begin{pmatrix} I_d\\m^\ast \end{pmatrix}=\begin{pmatrix} I_d\\ \tm^\ast \end{pmatrix}\tau^-
	\end{equation}
	
	As in the symplectic case, we can define the $q-$ function, we have the following properties for $q-$ function for $SHSp(2d)$ and $SU(d,d)$ to replace Lemma \ref{jac leb} and \eqref{ldq}.
	\begin{eqnarray}
	L^d(A_{\sigma+it})&=&\frac{1}{4d}\int_X-\ln q_{\sigma+it}(x)d\mu(x)\\
	|\frac{dm^+(\sigma+it, f(x))}{dm^+(\sigma+it, x)}|&=&|\det(\tau_{A_{\sigma+it}}(x))|^{-4d}
	\end{eqnarray}
	
	Now we can prove Theorem \ref{m+=m-} as the following. By non-increasing property of fibered rotation function, Lemma \ref{kotanimono} holds. And by the same proof as Lemma \ref{q compu}, we have
	\begin{equation}
	q^{-1}=e^{-4td^2}\cdot\Pi_ {i=1}^d(\frac{e^{4t}(1-\sigma_i(\tilde{m}^+)^2)}{1-e^{4t}\sigma_i(\tilde{m}^+)^2})^{2d}
	\end{equation}
	then we have the same inequality for $\frac{L^d}{t}$ as \eqref{l/t}.
	For the estimate of $\frac{\partial L^d}{\partial t}$ for Hermitian symplectic case, we use $\begin{pmatrix}
	I_d\\{m^-}^\ast
	\end{pmatrix}$ to replace $\begin{pmatrix}
	I_d\\ \overline{m^-}
	\end{pmatrix}$, $\tau^-$ defined in \eqref{mstar inv} to substitute $\overline{\tau_-}$ defined in Lemma \ref{conj m-}, use  \eqref{mstar inv} to replace Lemma \ref{conj m-}, we can get the following equation 
	\begin{equation}
	\int_X \mathfrak{R}(\text{tr}((I_d-\tm^{-^\ast} \tm^+)^{-1}(I_d+\tm^{-^\ast} \tm^+)))d\mu=\frac{\partial L^d}{\partial t}
	\end{equation}then the rest of the proof of Corollary \ref{main2} for $SHSp(2d), SU(d,d)$ cocycles is the same as the proof of Theorem \ref{mainresult}.\footnote{see appendix \ref{app A} for a discussion corresponds to the final part of the proof of Theorem \ref{generic+} for special Hermitian symplectic cocycles. }
	
	\subsection{The proof of Corollary \ref{main2}: $HSp(2d)$ and $U(d,d)$ cocycles}
	To prove corollary \ref{main2} for $HSp(2d)$ and $U(d,d)$ cocycles, we consider the following lemma which similar to Lemma \ref{zdense}:
	
	\begin{lemma}
		There exists a dense subset $\mathcal{Z}$ of $L^\infty(X, HSp(2d))$ satisfying all conditions of Theorem \ref{finitekotanitheory}. As a result, for all $A\in \mathcal{Z}, M(A)=0$.
	\end{lemma}
	
	\begin{proof}Consider a $HSp(2d)-$cocycle $A$ taking finitely many values, 
		Notice that there are cocycles  $B$ and $C$ also taking finitely many values and satisfying
		\begin{equation}\label{decomp U to SU time S1} 
		B(x)\in SHSp(2d), C(x)=c(x)\cdot I_{2d},A(x)=B(x)C(x)
		\end{equation} And the statement of Lemma \ref{zdense} also holds for the space of  $SHSp(2d)-$cocycle. So we can find a $SHSp(2d)-$cocycle $B'$, $L^\infty-$close to $B$ satisfying all conditions of Theorem \ref{finitekotanitheory}. 
		
		As a result, the $HSp(2d)-$cocycle $A':=B'C$ is $L^\infty-$close to $A$ and satisfying all conditions of Theorem \ref{finitekotanitheory}. In particular $M(A')=0$.
	\end{proof}
	The rest of the proof is the same as the part after Lemma \ref{zdense} for the proof of Theorem \ref{mainresult}. \footnote{see Appendix \ref{app A} for a discussion corresponds to the final part of Theorem \ref{generic+} for Hermitian symplectic cocycles. }
	\begin{rema}
	In general, for a $HSp(2d)-$cocycle $A(x)$, there is no cocycles $B,C$ in the same regularity class as $A$ and satisfying \eqref{decomp U to SU time S1}. So we can not use the result of $SHSp(2d)$ and $SU(d,d)$ cocycles directly to get the proof for $HSp(2d)$ and $U(d,d)$ cocycles. Moreover, since Lemma \ref{geometric meaning of tau hat} does not holds for general scalar matrix $A$, we can not prove corollary \ref{main2} for $HSp(2d)$ and $U(d,d)$ cocycles by mimicking the proof of Theorem \ref{mainresult} step by step.
	\end{rema}

	\subsection{Proof of Theorem \ref{mainresultjacobi}  and Corollary \ref{mainresultsch}}
	Firstly we consider the following classical result for stochastic Jacobi matrices on the strip in \eqref{jacobi strip} proved by B.Simon and S.Kotani (see \cite{kotani1988stochastic}).
	
	\begin{definition}
		For potential $v$ we define the following function,  $$M(v):=\text{the Lebesgue measure of }\{E\in \RR, L(A^{E-v})=0\}.$$And we say $v$ is deterministic if for $n\geq 0$, $v(f^n(x))$ is a measurable function of $\{v(f^k(x), k<0)\}$. 
	\end{definition}
	
	\begin{lemma}\label{simonkotanidet}
		Suppose $M(v)>0$, then $v$ is deterministic.
	\end{lemma}
	
	Combining Lemma \ref{simonkotanidet} with semi-continuity argument in subsection \ref{semicontinuity arg} (the proof is the same as in \cite{avila2005generic}, using harmonicity of $L^d$ on upper half plane), we can prove result similar to Theorem \ref{generic+} for stochastic Schr\"odinger operators and Jacobi matrices on the strips since non-periodic potentials which taking finitely many values are dense in $L^\infty(X, Her(d)), L^\infty(X, Sym_d\RR)$ and $L^\infty(X,\RR^d)$.\footnote{in that case where $\mu-$almost all points are periodic, see Appendix \ref{app B}.}
	
	Then using local regularization formula in \cite{avila2011density} (the proof is similar to Theorem \ref{poissoncontract}), we can get the proof of Theorem \ref{mainresultjacobi} and Corollary \ref{mainresultsch}.
	
	\appendix
	\section{Generic periodic (Hermitian) symplectic cocycles}\label{app A}
	In this section we finish the proof of Theorem \ref{generic+} (also for Hermitian symplectic cocycles) by considering the generic periodic cocycles. Without loss of generality, we consider the dynamics $(f,X,\mu)$ as the following: $f(x)=x+1$ for  $x\in X:=\ZZ/n\ZZ$, and for any $Z\subset X$,  $\mu(Z)=\frac{1}{n}\#(Z)$. 
	
	It is easy to see there is a set $\mathcal{O}\subset C(X,HSp(2d))=HSp(2d)^n$ such that $\mathcal{O}, \mathcal{O}\cap Sp(2d,\RR), \mathcal{O}\cap SHSp(2d)$ are residual sets in $HSp(2d),Sp(2d,\RR), SHSp(2d)$ respectively and satisfy the following property:  for all $\theta\in \RR/2\pi\ZZ$, the geometric and algebraic multiplicities of eigenvalues of $A_\theta^{(n)}(x)$ are $1$ and at most $2$ respectively. And there are only finite $\theta\in \RR/2\pi\ZZ$ such that $A_\theta^{(n)}(x)$ has repeated eigenvalues.
	
	We only need to prove the following lemma:
	
	\begin{lemma}
	There is a constant $C$ independent of $n$ such that for $A\in \mathcal{O}$, there exists $\theta\in (-\frac{C}{n}, \frac{C}{n})$ such that $L(A_\theta)>0$.
 	\end{lemma}
	
	\begin{proof}
	Without loss of generality, we only need to consider those $A$ with determinant equal to $1$. We fix such an $A\in \mathcal{O}$. Suppose there is an open interval $I$ containing $0$ and  $|I|>O(\frac{1}{n})$, such that 
	\begin{equation}\label{B1}
	\forall \theta\in I, \sigma(A_\theta^{(n)}(x)) \subset \{|z|=1\}
	\end{equation}where $|I|$ is the length of $I$. We prove the following lemma:
	
	\begin{lemma}\label{length of band}
		For any interval $I'$ such that 
		\begin{equation}
		\forall \theta\in I', \text{ all eigenvalues of } A_\theta^{(n)}(x) \text{ are simple and norm }1
		\end{equation}
		we have $$|I'|\leq O(\frac{1}{n})$$
	\end{lemma}
	
	\begin{proof}Suppose $I'=(a,b)$, for $\theta\in I'$, we conjugate it $\A_\theta^{(n)}$ to the matrix with form $$\diag(e^{i\rho_1(\theta)},\dots,e^{i\rho_d(\theta)},e^{-i\rho'_1(\theta)},\dots, e^{-\rho'_d(\theta)})$$
		where $\rho,\rho'$ depend continuously on $\theta$, $\sum_{i=1}^{d}\rho_i=\sum_{i=1}^{d}\rho'_i$ and $\rho_i+\rho'_j\notin 2\pi \ZZ$.
		
		Therefore $\rho_i'(b-)+\rho_i(b-)-\rho_i(a+)-\rho_i'(a+)<4\pi$. As a result, by the definition and property of fibered rotation function $\rho$, we have 
		\begin{eqnarray}\label{rho var interval}
			\rho(a)-\rho(b)&=&\frac{1}{n}\sum_{i=1}^{d}\rho_i'(b-)-\frac{1}{n}\sum_{i=1}^{d}\rho_i'(a+)\\
			&=&\frac{1}{2n}\sum_{i=1}^{d}\rho_i'(b-)+\rho_i(b-)-\rho_i'(a+)-\rho_i(a+)\nonumber\\
			&\leq & \frac{2d\pi}{n}\nonumber
		\end{eqnarray}
		But as in the proof of Lemma \ref{kotanimono} and \eqref{l/t}, for almost all $\theta\in I'$, we have 
		\begin{equation}\label{A2}
		-\frac{d\rho}{d\theta}=\lim_{t\to 0^+}\frac{\partial L^d(A_{\theta+it})}{\partial t}=\lim_{t\to 0}\frac{L^d(A_{\theta+it})}{t}\geq d
		\end{equation}
		Combined with  \eqref{rho var interval}, we get the bound of $|I'|$.
	\end{proof}
	
	As a result, there are two intervals $I_1, I_2\subset I$ satisfying \eqref{B1} and sharing common boundary point $\theta_0$ such that $\det(\lambda I_d-A_{\theta_0}^{(n)}(x))$ has a double root $e^{i\rho(\theta_0)}$.  Denote the corresponding general eigenspace\footnote{A general eigenspace of eigenvalue $\lambda$ of linear transform $A$ is defined by $V_\lambda:=\{v: \exists n, (A-\lambda)^n\cdot v=0\}$} by $V_{\theta_0}$.

To state the next lemma and for later use we introduce some basic concepts of Hermitian symplectic geometry on $\CC^{2d}$ (see for example \cite{harmer2007hermitian}).
\begin{definition}
	The two-form $<\cdot,\cdot>$ is linear in the second argument and conjugate
	linear in the first argument, is a hermitian symplectic form if $$<\phi,\psi>=-\overline{<\psi,\phi>}$$
\end{definition}

Let $e_i$ be the standard basis of $\CC^{2d}$, then we can define a classical Hermitian symplectic form $<\cdot,\cdot>$ such that $<e_i,e_{d+i}>=1$ for $1\leq i\leq d$ and $<e_i,e_j>=0$ if $|i-j|\neq d$. The Hermitian symplectic groups are all transformation preserving this form. A basis satisfying the same relation is called a \textit{canonical basis}. A subspace $V$ of $\CC^{2d}$ is called a \textit{Hermitian symplectic subspace } if $<\cdot,\cdot>|_V$ is non-degenerate. If an even dimensional Hermitian symplectic subspace $V$ admits a canonical basis then we say $V$ is canonical. A subspace $N$ is called \textit{isotropic} if $N\subset N^\perp:=\{v, <v,w>=0, \forall w\in N \}$. For a $2l-$dimensional Hermitian symplectic subspace $V$, if there is an $l-$dimensional isotropic subspace $N\subset V$, then we call $N$ is a \textit{Lagrange plane} of $V$. It can be proved that a Hermitian symplectic subspace $V$ contains a Lagrange plane if and only if it is canonical (see \cite{harmer2007hermitian}).

Now we state the lemma:
	\begin{lemma}
		By a suitable choice of canonical basis, we can assume $A_{\theta_0}^{(n)}, A_{\theta_0}^{(n)}|_{V_{\theta_0}}$ to be \begin{eqnarray}\label{B3}
		\begin{pmatrix}
		e^{i\rho(\theta_0)}&&\beta(\theta_0)&\\
		&\ast_{(d-1)\times (d-1)}&&\ast_{(d-1)\times (d-1)}\\
		&&e^{i\rho(\theta_0)}&\\
		&\ast_{(d-1)\times (d-1)}&&\ast_{(d-1)\times (d-1)}
		\end{pmatrix}\\ 
		\begin{pmatrix}
		e^{i\rho(\theta_0)}&&\beta(\theta_0)&\\
		&0_{(d-1)\times (d-1)}&&0_{(d-1)\times (d-1)}\\
		&&e^{i\rho(\theta_0)}&\\
		&0_{(d-1)\times (d-1)}&&0_{(d-1)\times (d-1)}\label{B4}
		\end{pmatrix}
		\end{eqnarray}respectively, where $\beta(\theta_0)\neq 0$.
	\end{lemma}
	\begin{proof}
Notice that different general eigenspaces are mutually orthogonal (in the Hermitian symplectic sense). So $V_{\theta_0}$ is a Hermitian symplectic subspace of $\CC^{2d}$.
		
		Moreover consider the only one eigenvector $v\in V_{\theta_0}-\{0\}$ and $w\in V_{\theta_0}- \CC\cdot v$, then
		\begin{equation}
		A\cdot v=e^{i\rho(\theta_0)}v, A\cdot w-e^{i\rho(\theta_0)}w\in \CC\cdot v-\{0\}
		\end{equation}
		Therefore computing $<Av,Aw>-<v,w>$ we get 
		\begin{equation}
		<v,v>=0
		\end{equation}
		Since $V_{\theta_0}$ contains a Lagrange plane, it is a canonical Hermitian symplectic subspace of $\CC^{2d}$. Then we can pick a vector $v'\in V_{\theta_0}-\CC\cdot v$ such that $<v',v'>=0$ and by further normalization we can assume $<v,v'>=1$.
		
		It is easy to prove that $V_{\theta_0}^\perp$ is also a canonical Hermitian symplectic subspace. Then we can extend $v,v'$ to a canonical basis of $\CC^{2d}$. Then $A_{\theta_0}^{(n)}, A_{\theta_0}^{(n)}|_{V_{\theta_0}}$ have the form in \eqref{B3} with respect to this canonical basis.
	\end{proof}
Choose a contour $\Gamma$ enclosing $e^{i\rho(\theta_0)}$ but no other points of $\sigma(A_{\theta_0}^{(n)})$, then for $\theta$ close to $\theta_0$, $\Gamma$ encloses exactly two points of $\sigma(A_{\theta}^{(n)})$. Therefore there is a $2-$dimensional invariant space $V_\theta $ for eigenvalues of $A_\theta^{(n)}$ contained in the bounded region with  boundary $\Gamma$. 

Consider the spectral projection  $\frac{1}{2\pi i}\int_{\Gamma}\frac{1}{z-A_\theta^{(n)} }dz$ of $A_\theta^{(n)}$ onto $V_\theta$. Since $\sigma(A_{\theta}^{(n)})\subset \{|z|=1\}$, we have (by suitable choice of branch of $z\mapsto \sqrt{z}$): 
\begin{eqnarray}\label{B7}
\frac{\text{tr}(\frac{1}{2\pi i}\int_{\Gamma}\frac{z}{z-A_\theta^{(n)} }dz)}{\sqrt{\text{tr}(\Lambda^2(\frac{1}{2\pi i}\int_{\Gamma}\frac{z}{z-A_\theta^{(n)} }dz))}}&\leq& 2\\
\frac{\text{tr}(\frac{1}{2\pi i}\int_{\Gamma}\frac{z}{z-A_{\theta_0}^{(n)} }dz)}{\sqrt{\text{tr}(\Lambda^2(\frac{1}{2\pi i}\int_{\Gamma}\frac{z}{z-A_{\theta_0}^{(n)} }dz))}}&=& 2\label{B8}
\end{eqnarray}
Therefore 
\begin{equation}
\frac{d}{d\theta}|_{\theta=\theta_0}\frac{\text{tr}(\frac{1}{2\pi i}\int_{\Gamma}\frac{z}{z-A_\theta^{(n)} }dz)}{\sqrt{\text{tr}(\Lambda^2(\frac{1}{2\pi i}\int_{\Gamma}\frac{z}{z-A_\theta^{(n)} }dz))}}=0
\end{equation}
Then we get 
\begin{eqnarray}\label{B10}
&& 2\frac{d}{d\theta}|_{\theta=\theta_0}\text{tr}(\frac{1}{2\pi i}\int_{\Gamma}\frac{z}{z-A_{\theta}^{(n)} }dz)\cdot \text{tr}(\Lambda^2(\frac{1}{2\pi i}\int_{\Gamma}\frac{z}{z-A_{\theta_0}^{(n)} }dz))\\
&=&\text{tr}(\frac{1}{2\pi i}\int_{\Gamma}\frac{z}{z-A_{\theta_0}^{(n)} }dz)\cdot \frac{d}{d\theta}|_{\theta=\theta_0}\text{tr}(\Lambda^2(\frac{1}{2\pi i}\int_{\Gamma}\frac{z}{z-A_{\theta}^{(n)} }dz))\nonumber
\end{eqnarray}

By \eqref{B4}, we have 
\begin{eqnarray}
\text{tr}(\Lambda^2(\frac{1}{2\pi i}\int_{\Gamma}\frac{z}{z-A_{\theta_0}^{(n)} }dz))&=& e^{2i\rho(\theta_0)}\label{B11}\\
\text{tr}(\frac{1}{2\pi i}\int_{\Gamma}\frac{z}{z-A_{\theta_0}^{(n)} }dz)&=& 2e^{i\rho(\theta_0)}\label{B12}
\end{eqnarray}

Since $\frac{d}{d\theta}|_{\theta=\theta_0}A_{\theta}^{(n)}{A_{\theta_0}^{(n)}}^{-1}\in \mathfrak{shsp}(2d)$, we can assume 
\begin{equation}\label{B13}
\frac{d}{d\theta}|_{\theta=\theta_0}A_{\theta}^{(n)}{A_{\theta_0}^{(n)}}^{-1}=\begin{pmatrix}
X&Y\\
Z&-X^\ast
\end{pmatrix}
\end{equation}where $Y=Y^\ast, Z=Z^\ast$ and $\text{tr}(X-X^\ast)=0$. Let $X,Y,Z$ be $(x_{ij})_{ i,j=1}^d, (y_{ij})_{i,j=1}^d, (z_{ij})_{i,j=1}^d$ respectively.

By computation we get,
\begin{equation}
\frac{d}{d\theta}|_{\theta=\theta_0}\text{tr}(\frac{1}{2\pi i}\int_{\Gamma}\frac{z}{z-A_{\theta}^{(n)} }dz)\\
=\text{tr}(\frac{d}{d\theta}|_{\theta=\theta_0}A_{\theta} \cdot \frac{1}{2\pi i}\int_{\Gamma}\frac{z}{(z-A_{\theta_0}^{(n)})^2}dz)
\end{equation}

Notice that $\frac{1}{2\pi i}\int_{\Gamma}\frac{z}{(z-A_{\theta_0}^{(n)})^2}dz$ is the spectral projection of $A_{\theta_0}^{(n)}$ onto $V_{\theta_0}$, we have
\begin{eqnarray}\label{B14}
&&\frac{d}{d\theta}|_{\theta=\theta_0}\text{tr}(\frac{1}{2\pi i}\int_{\Gamma}\frac{z}{z-A_{\theta}^{(n)} }dz)\\
&=&\text{tr}(\frac{d}{d\theta}|_{\theta=\theta_0}A_{\theta}^{(n)}{A_{\theta_0}^{(n)}}^{-1}\cdot A_{\theta_0}^{(n)}|_{V_{\theta_0}})\nonumber\\
&=&e^{i\rho(\theta_0)}(x_{11}-\overline{x_{11}})+\beta(\theta_0)z_{11}\nonumber
\end{eqnarray}

Denote $A_{\theta}^{(n)}|_{V_{\theta}}$ by $B_\theta=(b_{i,j}(\theta))_{1\leq i,j\leq d}$, then
\begin{eqnarray}\label{B15}
&&\frac{d}{d\theta}|_{\theta=\theta_0}\text{tr}(\Lambda^2(\frac{1}{2\pi i}\int_{\Gamma}\frac{z}{z-A_{\theta}^{(n)} }dz))\\
&=&\frac{d}{d\theta}|_{\theta=\theta_0}\text{tr}(\Lambda^2(A_{\theta}^{(n)}|_{V_{\theta}})) \nonumber\\
&=&\frac{d}{d\theta}|_{\theta=\theta_0}\text{tr}(\Lambda^2(B_\theta)) \nonumber\\
&=&\sum_{i<j}\begin{vmatrix}b'_{i,i}(\theta_0)& b'_{i,j}(\theta_0)\\
b_{j,i}(\theta_0)&b_{j, j}(\theta_0)
\end{vmatrix}+\begin{vmatrix}b_{i,i}(\theta_0)& b_{i,j}(\theta_0)\\
b'_{j,i}(\theta_0)&b'_{j, j}(\theta_0)
\end{vmatrix} \nonumber\\
&=&\begin{vmatrix}b'_{1,1}(\theta_0)& b'_{1,d+1}(\theta_0)\\
b_{d+1,1}(\theta_0)&b_{d+1, d+1}(\theta_0)
\end{vmatrix}+\begin{vmatrix}b_{1,1}(\theta_0)& b_{1,d+1}(\theta_0)\\
b'_{d+1,1}(\theta_0)&b'_{d+1, d+1}(\theta_0)
\end{vmatrix} \nonumber\\
&&(\text{ since } b_{i,j}(\theta_0)=0 \text{ except } \{i,j\}\subset \{1, d+1\})\nonumber
\end{eqnarray}

By \eqref{B3},\eqref{B4},\eqref{B13} we have $$\begin{pmatrix}
b_{1,1}(\theta_0)&b_{1,d+1}(\theta_0)\\
b_{d+1,1}(\theta_0)&b_{d+1,d+1}(\theta_0)
\end{pmatrix}=\begin{pmatrix}
e^{i\rho(\theta_0)}&\beta(\theta_0)\\
&e^{i\rho(\theta_0)}
\end{pmatrix}$$
$$\begin{pmatrix}
b'_{1,1}(\theta_0)&b'_{1,d+1}(\theta_0)\\
b'_{d+1,1}(\theta_0)&b'_{d+1,d+1}(\theta_0)
\end{pmatrix}=
\begin{pmatrix}e^{i\rho(\theta_0)}x_{11}&\beta(\theta_0)x_{11}+e^{i\rho(\theta_0)}y_{11}
\\e^{i\rho(\theta_0)}z_{11}&\beta(\theta_0)z_{11}-e^{i\rho(\theta_0)}\overline{x_{11}}
\end{pmatrix}$$
Then combining with \eqref{B15} we get 
\begin{equation}\label{B16}
\frac{d}{d\theta}|_{\theta=\theta_0}\text{tr}(\Lambda^2(\frac{1}{2\pi i}\int_{\Gamma}\frac{z}{z-A_{\theta}^{(n)} }dz))=e^{2i\rho(\theta_0)}(x_{11}-\overline{x_{11}})
\end{equation}

Combining \eqref{B16},\eqref{B14},\eqref{B11},\eqref{B12},\eqref{B10} we get 
\begin{equation}\label{z11=0}
z_{11}=0
\end{equation}

	Now we define the \textit{monotonic} special Hermitian symplectic cocycles: consider a cone $\mathcal{C}$ on $\mathfrak{shsp}(2d,\RR)$ defined by \begin{equation}
	\mathcal{C}:=\{W\in \mathfrak{shsp}(2d,\RR): J\cdot W \text{ is negative definite}\}
	\end{equation}where $J$ is defined in Definition \ref{sympl group def}. Obviously for any $W_1,W_2\in \mathcal{C}$, $W_1+W_2\in \mathcal{C}$.
	\begin{definition}
	Let $\theta\mapsto D_\theta\in C(X,SHSp(2d,\RR))$ be a one parameter family of continuous symplectic cocycles and $C^1$ in $\theta$. We say it is monotonic at $\theta_0$ if \begin{equation}
	\frac{d}{d\theta}|_{\theta=\theta_0}D_{\theta}(x){D_{\theta_0}^{-1}}(x)\in \mathcal{C}
	\end{equation}for any $x\in X$.
	\end{definition}
	We have:
	\begin{lemma}\label{mono iter mono}
	If the one parameter family of symplectic cocycle $D_\theta$ is monotonic at $\theta_0$, then for any $n>0$, $D_\theta^{(n)}$ is also monotonic at $\theta$.
	\end{lemma}
	\begin{proof}
	First of all we prove the invariance of the cone $\mathcal{C}$ under inner automorphism.
	 \begin{lemma}\label{inv cone}
	 For any $g\in SHSp(2d,\RR)$, $g\mathcal{C}g^{-1}=\mathcal{C}$
	 \end{lemma}
	 \begin{proof}Notice that for any $g\in SHSp(2d,\RR)$, $Jg=(g^{-1})^\ast J$.
	 Then for any $W\in \mathcal{C}$, 
	 \begin{eqnarray}
	 JgWg^{-1}=(g^{-1})^\ast JWg^{-1}
	 \end{eqnarray}
	 is negative definite (since $g$ is invertible and $JW$ is negative definite). Therefore we have for any $g$, $g\mathcal{C}g^{-1}\subset\mathcal{C}$. It is easy to prove that the equality holds.
	 \end{proof}
	 Suppose $D_\theta$ is monotonic at $\theta_0$, then for any $x\in X$, using Lemma \ref{inv cone}, we have
	 \begin{eqnarray}
	&&\frac{d}{d\theta}|_{\theta=\theta_0}D^{(n)}_{\theta}(x) {D^{(n)}_{\theta_0}}^{-1}(x)\\
	 &=&\sum_{i=0}^{n}D_{\theta_0}(f^{n-1}(x))\cdots D_{\theta_0}(f^{i}(x))\cdot (\frac{d}{d\theta}|_{\theta=\theta_0}D_\theta(f^i(x))D_{\theta_0}(f^i(x))^{-1})\nonumber\\
	 &&\cdot(D_{\theta_0}(f^{n-1}(x))\cdots D_{\theta_0}(f^{i}(x)))^{-1}\nonumber
	 \in \mathcal{C}
	 \end{eqnarray}which implies $D_\theta^{(n)}$ is also monotonic at $\theta_0$.
	\end{proof}

Come back to the discussion of periodic cocycle $A_\theta$. It is easy to check that the one parameter family cocycles $\theta\mapsto A_\theta$ is monotonic on $\RR$, then by Lemma \ref{mono iter mono}, $\theta\mapsto A_\theta^{(n)}$ is monotonic at $\theta_0$. As a result, $J\cdot\frac{d}{d\theta}|_{\theta=\theta_0}A_{\theta}^{(n)}{A_{\theta_0}^{(n)}}^{-1}$ is negative definite, using \eqref{B13} we get $Z$ is negative definite, which contradicts with \eqref{z11=0}. In summary, $I_1, I_2$ cannot have common boundary point. 
\end{proof}

\section{Generic periodic Schr\"odinger operator and Jacobi matrices on the strip}\label{app B}
For periodic Sch\"odinger operator and Jacobi matrices on the strip, consider the corresponding (Hermitian) symplectic cocycle: $$A^{(E-v)}: \ZZ/n\ZZ\to Sp(2d,\RR) \text{ or } SHSp(2d,\RR)$$where $A^{(v)}(x):=\begin{pmatrix}
v(x)&-I_d\\
I_d
\end{pmatrix}, v(x)\in \RR^d\hookrightarrow Sym_d\RR, Sym_d\RR$ or $her(d)$ is the corresponding  potential.   As in Appendix \ref{app A}, we hope to prove the following lemma:
\begin{lemma}\label{per sch}
There is a constant $C$ independent of  $n$ such that for generic potential $v$ in $(\RR^d)^n$, $(Sym_d\RR)^n$ or $her(d)^n$, there exists $E\in (-\frac{C}{n}, \frac{C}{n})$ such that $L(A^{(E-v)})>0$.
\end{lemma}
\begin{proof}
At first we prove the following fact:
	\begin{lemma}\label{sch mono}
		For any Jacobi matrices on the strip, the corresponding one parameter cocycle $E\mapsto (A^{(E-v)})^{(2)}$ is monotonic on $\RR$.
	\end{lemma}
	\begin{proof}
		By computation 
		\begin{eqnarray*}\nonumber
			&& J\frac{d}{dE}|_{E=E_0}(A^{(E_0-v)}(f(x))A^{(E_0-v)}(x))\cdot (A^{(E_0-v)}(x))^{-1}(A^{(E_0-v)}(f(x))^{-1}\\
			&=&\begin{pmatrix}
				-I_d&E-v(f(x))\\
				E-v(f(x))&-I_d-(E-v(f(x)))^2\end{pmatrix}\\
			&=&-\begin{pmatrix}
				I_d&\\
				-(E-v(f(x)))&I_d
			\end{pmatrix}\cdot \begin{pmatrix}
			I_d&-(E-v(f(x)))\\
			&I_d\end{pmatrix}
	\end{eqnarray*}which is negative definite, by definition of monotonicity in Appendix \ref{app A}, we get the proof.
\end{proof}
 Denote $B_E(x):=(A^{(E-v)})^{(2)}$. Since $E\mapsto B_E(x)$ is monotonic, mimic the proof of Theorem \ref{rot num}, we can define the $m-$function (taking values in $\{Z,\norm{Z}<1\}$), fibered rotation function (which is non increasing) and complexify the Lyapunov exponent $E+it \mapsto (L^d+i\rho)(B_{E+it})$. In fact these properties are already known for Jacobi matrices on the strip with form in \eqref{jacobi strip}, see \cite{kotani1988stochastic}, \cite{rotation2007schulz} for example or \cite{liu2015monotonic} for general monotonic symplectic cocycles.

The proof of Lemma \ref{per sch} is basically the same as the discussion in Appendix \ref{app A}. At first we prove the following Lemma similar to Lemma \ref{length of band}.

\begin{lemma}\label{length of band scho}
	For any interval $I'$ such that 
	\begin{equation}
	\forall E\in I', \text{ all eigenvalues of } B_E^{(n)}(x) \text{ are simple and norm }1
	\end{equation}
	we have $$|I'|\leq O(\frac{1}{n})$$
\end{lemma}

\begin{proof}
As in the proof of Lemma \ref{length of band}, we only need to prove for almost all $E\in I'$, we have 
\begin{equation}
\limsup_{t\to 0^+}\frac{ L^d(B_{E+it})}{ t}\geq c(d)
\end{equation}where $c(d)$ is a constant independent of $n$. Let $$\hat{B}_{E+it}(x):=C(f^2(x))B_{E+it}(x)C(x)^{-1}.$$where $C(x):=\begin{pmatrix}
I_d&E-v(f^{-1}(x))\\&I_d
\end{pmatrix}$. Then all the dynamical properties of cocycles $B_{E+it}$ and $\hat{B}_{E+it}$ are the same. Moreover by computation, 
\begin{equation}
\frac{\partial\hat{B}_{E+it}(x)}{\partial t}=\begin{pmatrix}
&i\cdot I_d\\
-i\cdot I_d&
\end{pmatrix}\cdot \hat{B}_{E+it}(x)
\end{equation}
Therefore we have \begin{equation}
\nb_{E+it}(x)=(\begin{pmatrix}
e^{-t}&\\&e^t
\end{pmatrix}+o(t))\cdot \nb_{E}(x)
\end{equation}
where $\nb(x):=C\hat{B}_E(x)C^{-1}$, $C$ is the Cayley element defined in \eqref{Cayley ele}. 

Without loss of generality, we fix an $E$ such that $\lim_{t\to0^+}\Phi_C^{-1}\cdot \hat{m}(E+it, x)$ exists and be finite, where $\hat{m}(E+it, \cdot)$ is the associated $m-$function of cocycle $\nb_{E+it}$. Since $\Phi_C^{-1}\cdot \hat{m}$ is the classical $m-$function taking values in $\{Z,Im(Z)>0\}$ for Jacobi matrices on the strip (see \cite{kotani1988stochastic} for example), by Lemma \ref{herglotz} we know that almost every $E\in \RR$ satisfies our condition.  

As Definition \ref{q def}, we define by $\hat{q}_{E+it}(x)$ the Jacobian (with respect to the volume form induced by the Bergman metric) of the map $$Z\mapsto \nb_{E+it}(x)\cdot Z$$
at point $Z=\hat{m}(E+it, x)$. Then similar to the case of symplectic cocycles, we have \begin{equation}\label{ld q hat exp}
L^d(\nb_{E+it})=\frac{1}{4d}\int_X\ln \hat{q}_{E+it}^{-1}d\mu . 
\end{equation} Moreover as the proof of Lemma \ref{q compu}, we have 
\begin{eqnarray}
\hat{q}_{E+it}^{-1}(x)&=&\frac{V(e^{2t}\hat{m}(E+it, f^2(x))+o(t))}{V(\hat{m}(E+it, f^2(x)))}\cdot e^{4td^2+o(t)}\nonumber\\
&=& e^{4td^2+o(t)}\cdot \Pi_{i=1}^d\frac{1-\sigma_i^2(\hat{m}(E+it, f^2(x)))}{1-(e^{4t} \sigma_i^2(\hat{m}(E+it, f^2(x)))+o(t))}\label{q hat inv expr}
\end{eqnarray}
By our choice of $E$, $\lim_{t\to 0^+}\sigma_i^2(\hat{m}(E+it, f^2(x)))$ exists and not greater than $1$. Therefore using the inequality \eqref{ele ln ineq}, we have 
\begin{equation}
 \frac{1-\sigma_i^2(\hat{m}(E+it, f^2(x)))}{1-(e^{4t} \sigma_i^2(\hat{m}(E+it, f^2(x)))+o(t))}\geq 1 \text{, when $t$ is small }
\end{equation}
Combining with \eqref{q hat inv expr}\eqref{ld q hat exp} we get,
\begin{equation}
L^d(\nb_{E+it})\geq dt+o(t) \text{, when $t$ is small }
\end{equation}for almost all $E\in I'$, which implies our lemma.
\end{proof}

As a result,  to prove Lemma \ref{per sch}, we only need to prove for generic potential $v$, two intervals $I_1,I_2$ satisfying the condition of Lemma \ref{length of band scho} cannot share a common boundary point. But in the corresponding part of Appendix \ref{app A}, we only use the condition that the one parameter family $\theta\mapsto A_\theta$ is monotonic, then by Lemma \ref{sch mono}, the proof is the same here.
\end{proof}

\end{document}